\documentclass[12pt]{amsart}
\usepackage[T2A]{fontenc}
\usepackage[english]{babel}
\usepackage{amsaddr}
\usepackage{amsmath}
\usepackage{amssymb}
\usepackage{amsfonts}
\usepackage{epsfig}
\usepackage{srcltx}
\usepackage{subfigure}
\usepackage{cite}
\usepackage{float}
\usepackage{mathtools}
\usepackage[a4paper,  mag=1000, includefoot,  left=2cm, right=2cm, top=2cm, bottom=2cm, headsep=1cm, footskip=1cm ]{geometry}
\usepackage[unicode,colorlinks]{hyperref}
\hypersetup{colorlinks=true, citecolor=blue, linkcolor=blue}

\newtheorem{Th}{Theorem}
\newtheorem{Lem}{Lemma}

\newtheorem{Cor}{Corollary}

\begin{document}
\thispagestyle{empty}

\title[]{Resonances in asymptotically autonomous systems with a decaying chirped-frequency excitation}
\author{Oskar A. Sultanov}

\address{
Chebyshev Laboratory, St. Petersburg State University, 14th Line V.O., 29, Saint Petersburg 199178 Russia.}
\email{oasultanov@gmail.com}

%\thanks{\it \today}

\maketitle
{\small

{\small
\begin{quote}
\noindent{\bf Abstract.} The influence of oscillatory perturbations on autonomous strongly nonlinear systems in the plane is investigated. It is assumed that the intensity of perturbations decays with time, and their frequency increases according to a power law. The long-term behaviour of perturbed trajectories is discussed. It is shown that, depending on the structure and the parameters of perturbations, there are at least two different asymptotic regimes: a phase locking and a phase drifting. In the case of phase locking, resonant solutions with an unlimitedly growing energy occur. The stability and asymptotics at infinity of such solutions are investigated. The proposed analysis is based on a combination of the averaging technique and the method of Lyapunov functions.

\medskip


%	37J65  	Nonautonomous Hamiltonian dynamical systems (Painleve equations, etc.)
%	34D10 Perturbations of ordinary differential equations
% 34D20  	Stability of solutions to ordinary differential equations
%	34C15  	Nonlinear oscillations and coupled oscillators for ordinary differential equations
\end{quote}
}

\section{Introduction}

In this paper, the effect of oscillating perturbations on autonomous Hamiltonian systems is investigated. It is assumed that the intensity of perturbations decays with time and the limiting system describes strongly nonlinear oscillations. The conditions of existence and stability of resonant solutions with growing energy are discussed.

Qualitative properties of solutions for asymptotically autonomous systems have previously been studied in many papers. It is known, in particular, that if the limiting system is asymptotically stable, the trajectories of the perturbed system remain in some neighbourhood of a stable solution~\cite{RB53,LM56}. See also~\cite{LDP74}, where the conditions were described under which decaying nonlinear disturbances do not violate the behaviour of solutions to linear autonomous oscillating systems. In the general case, the long-term behaviour of solutions to asymptotically autonomous systems can differ from the dynamics described by the corresponding autonomous systems~\cite{HRT94}. It depends on the qualitative properties of solutions to the limiting system and the structure of decaying perturbations~\cite{HRT92,LRS02,KS05,MR08,LH19}.

Asymptotically autonomous systems with oscillatory decreasing perturbations have been studied in several papers, where either linear equations were considered~\cite{DF78,AK96,PN06,BN10,ML14}, or the behaviour of solutions in some vicinity of the equilibrium was discussed~\cite{OS20}. In the present paper, the effect of damped oscillatory disturbances on nonlinear systems far from equilibrium is investigated.

Note that the influence of small oscillatory perturbations on dynamical systems is well-studied problem~\cite{BM61,Hap93,AKN06,VB07,GKT17}. In particular, chirped-frequency perturbations with a small parameter are effectively used to control the dynamics of nonlinear systems~\cite{LFJPA08,LKRMS08,GMetal17,LFAGS20,AK20}. However, in this paper the presence of a small parameter is not assumed. We consider strongly nonlinear systems in the plane with chirped-frequency oscillatory perturbations vanishing at infinity in time. To the best of our knowledge, bifurcations in such systems have not been thoroughly investigated.

The paper is organized as follows. In section~\ref{sec1}, the mathematical formulation of the problem is given and the class of decreasing perturbations is described. First we construct a change of variables that simplifies the perturbed system in the leading asymptotic terms. The construction of this transformation is described in section~\ref{sec2}. Depending on the structure of the simplified equations there are at least two asymptotic regimes for solutions of the perturbed system: a phase locking and a phase drifting. Such regimes are described in section~\ref{sec3}. Nonlinear stability analysis of the phase locking is discussed in section~\ref{sec4}. In section~\ref{secEX}, the proposed theory is applied to the examples of non-autonomous systems with decaying oscillatory perturbations. The paper concludes with a brief discussion of the results obtained.

\section{Problem statement}
\label{sec1}
Consider the asymptotically autonomous system in the plane:
\begin{gather}
\label{FulSys}
\frac{dx}{dt}=\partial_y H(x,y)+t^{-\frac{a}{q}} f(x,y,S(t),t), \quad \frac{dy}{dt}=-\partial_x H(x,y)+t^{-\frac{a}{q}} g(x,y,S(t),t), \quad t>0,
\end{gather}
with $S(t)=s t^{1+b/q}$ and the parameters $a,b,q\in\mathbb Z$, $s\in\mathbb R$ such that $1\leq a,b\leq q$, $s>0$. It is assumed that the functions $H(x,y)$, $f(x,y,S,t)$ and $g(x,y,S,t)$, defined for all $(x,y,S)$ in $\mathbb R^3$, $t>0$, are infinitely differentiable and  $2\pi$-periodic functions with respect to $S$. The Hamiltonian $H(x,y)$ of the corresponding limiting autonomous system
\begin{gather}
    \label{LimSys}
        \frac{dx}{dt}=\partial_y H(x,y), \quad \frac{dy}{dt}=-\partial_x H(x,y)
\end{gather}
is assumed to have the following form
\begin{gather*}
H(x,y)= \frac{y^2}{2}+U(x), \quad U(x)=\frac{x^{2h}}{2h} +\sum_{i=1}^{2h-1}u_i x^i
\end{gather*}
with $h\in\mathbb Z$, $h\geq 2$ and $u_i={\hbox{\rm const}}$. In this case, there exists $E_0>0$ such that for all $E>E_0$ the level lines $\{(x,y)\in\mathbb R^2: H(x,y)=E\}$ are closed curves on the phase space $(x,y)$ parameterized by the parameter $E$ and do not contain any fixed points of system \eqref{LimSys}. Let $x_-(E)<0<x_+(E)$ be the solutions of the equation $U(x)= E$ with $E>E_0$. Then, to each closed curve there correspond a periodic solution $x_0(t,E)$, $y_0(t,E)$ of system \eqref{LimSys} with a period
\begin{gather*}
    T(E)\equiv \int\limits_{x_-(E)}^{x_+(E)} \frac{\sqrt 2 d\varsigma}{\sqrt{E-U(\varsigma)}}=\kappa E^{\frac{1-h}{2h}} \big(1+\mathcal O(E^{-\frac{1}{2h}})\big), \quad
    E\to\infty, \quad \kappa =\sqrt2 (2h)^{\frac{1}{2h}} \int\limits_{-1}^{1} \frac{d\varsigma}{\sqrt{1-\varsigma^{2h}}}.
\end{gather*}

The perturbations of the autonomous system \eqref{LimSys} are described by the functions with power-law asymptotics:
\begin{gather}
\label{fg}
    f(x,y,S,t)=\sum_{k=0}^\infty t^{-\frac{k}{q}} f_k(x,y,S), \quad
    g(x,y,S,t)=\sum_{k=0}^\infty t^{-\frac{k}{q}} g_k(x,y,S), \quad t\to\infty,
\end{gather}
with
\begin{align*}
        f_k(x,y,S)= \sum_{i=0}^{p} \sum_{j=0}^{l-1}A_{k,i,j}(S) x^{i} y^j, \quad
         g_k(x,y,S)=\sum_{i=0}^{p} \sum_{j=0}^{l}  B_{k,i,j}(S) x^{i} y^j,
\end{align*}
where $l,p\in\mathbb Z$, $0\leq l\leq p\leq 2h-1$ and the coefficients $A_{k,i,j}(S)$, $B_{k,i,j}(S)$ are $2\pi$-periodic with respect to $S$. It is assumed that $A_{k,i,j}(S)\equiv B_{k,i,j}(S)\equiv 0$ if $i+j>p$ and $A_{k,i,-1}(S)\equiv 0$ for all $k,i\geq 0$. The parameter $p$ is responsible for the maximum degree of the monomials $x^i y^j$ in the perturbations with nonzero coefficients, while the parameter $l$ corresponds to a maximum power of $y$. It is also assumed that these parameters satisfy the following inequalities:
\begin{gather}\label{rc}
-1\leq \sigma < \frac{b}{q}, \quad \sigma:=\frac{b}{q}\Big(l-1+\frac{p-1}{h-1}\Big)-\frac{a}{q}.
\end{gather}
The role of this condition will be specified below.

Note that decreasing perturbations with power-law asymptotics appear, for example, in the study of Painlev\'{e} equations~\cite{IKNF06,BG08}, phase-locking phenomena~\cite{LK14,DJD19,OS21}, stochastic perturbations~\cite{AGR09,OS19}, and in a wide range of other problems associated with nonlinear non-autonomous systems~\cite{BAD04,KF13}.

The simplest example is given by the perturbed Duffing oscillator:
\begin{gather}
\label{ex}
\frac{d^2x}{dt^2}-x+x^3=B t^{-\frac{1}{3}}\cos  \big(s t^{\frac 43}\big).
\end{gather}
It is readily seen that equation \eqref{ex} in the variables $x,y= \dot x$ takes form \eqref{FulSys} with $a=b=1$, $q=3$, $h=2$, $U(x)\equiv x^4/4-x^2/2$, and the perturbations $f\equiv 0 $, $g\equiv B\cos S$ satisfy \eqref{fg} and \eqref{rc} with $l=p=0$, $\sigma=-1$. Note that all trajectories of the corresponding limiting system \eqref{LimSys} are bounded, and the solutions with $H(x(t),y(t))\equiv E>0$ are periodic with the period $T(E)=\mathcal O( E^{-1/4})$ as $E\to\infty$ (see Fig.~\ref{Fig1}, a). Numerical analysis of equation \eqref{ex} with $B\neq 0$ shows that oscillatory decreasing perturbations can lead to the appearance of solutions with unboundedly growing energy $I(t)\equiv H(x(t),\dot x(t))$. Besides, for some values of the parameters, the solutions of the perturbed equation have the same long-term behaviour as the trajectories of the limiting system: $I(t)=\mathcal O(1)$ as $t\to\infty$ (see Fig.~\ref{Fig1}, b).

In this paper, we investigate the existence and stability of resonant solutions with unboundedly growing energy for system \eqref{FulSys} with decreasing oscillatory perturbations satisfying \eqref{fg} and \eqref{rc}.

\begin{figure}
\centering
\subfigure[ ]{\includegraphics[width=0.4\linewidth]{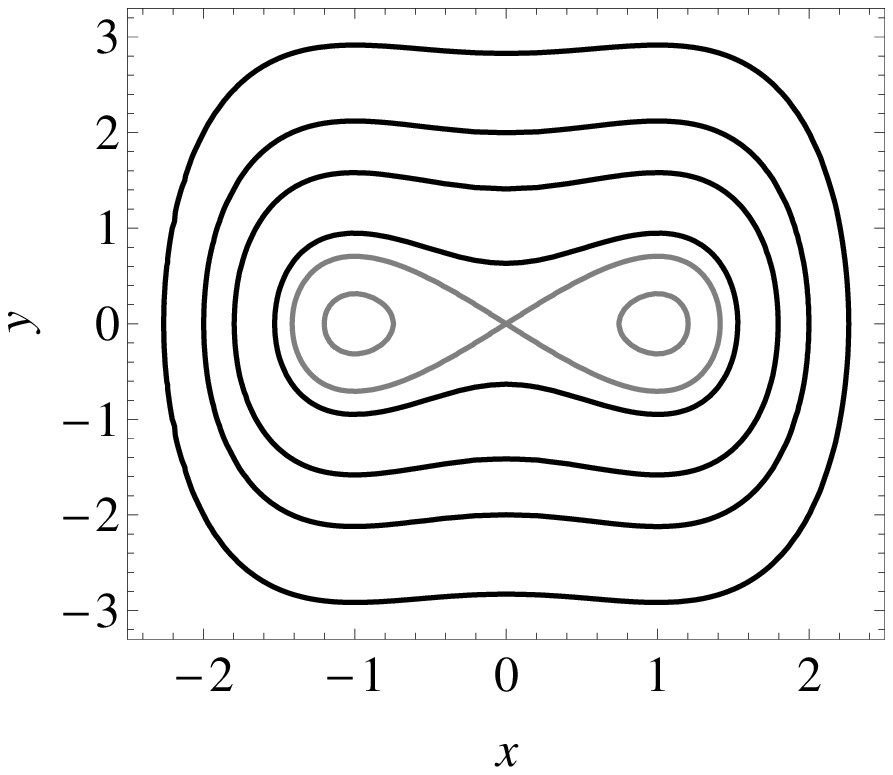}}
\hspace{4ex}
\subfigure[]{\includegraphics[width=0.4\linewidth]{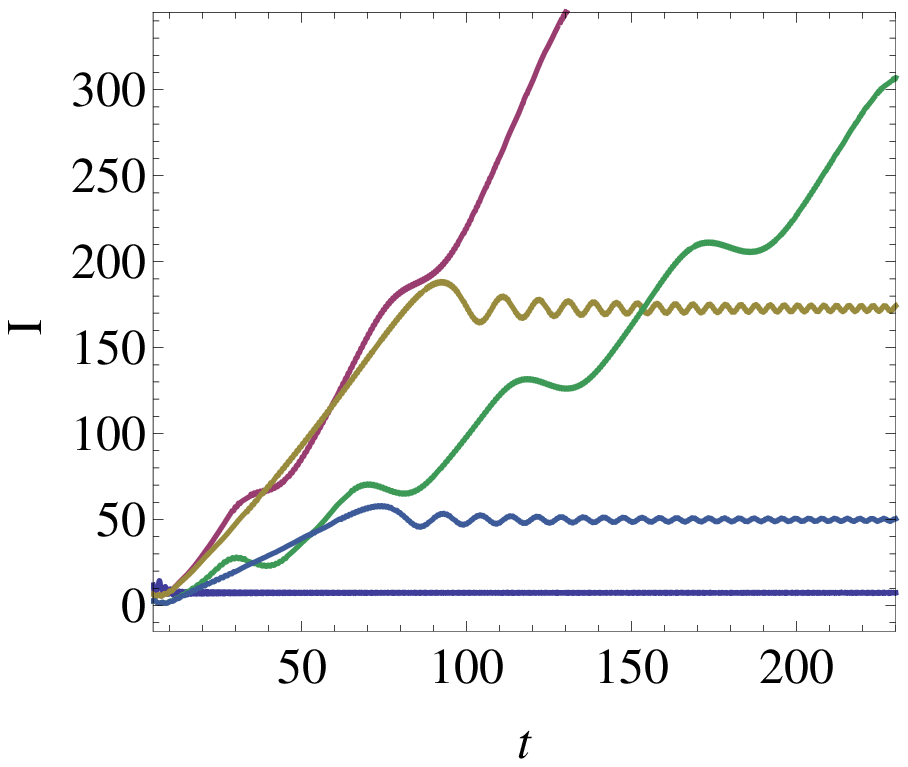}}
\caption{\footnotesize (a) The level lines of $H(x,y)$ for equation \eqref{ex} with $B=0$. The black curves correspond to solutions with $E>0$. The gray curves correspond to solutions with $E\leq 0$. (b) The evolution of $I(t)=H(x(t),\dot x(t))$ for solutions of \eqref{ex} with different values of the parameters $B,s$ and initial data.} \label{Fig1}
\end{figure}

\section{Change of variables}
\label{sec2}

In this section, we construct the transformations of variables that simplify system \eqref{FulSys}.

\subsection{Energy-angle variables}
First, define auxiliary $2\pi$-periodic functions $X(\phi,E)=x_0(\phi/\omega(E),E)$ and $Y(\phi,E)=y_0(\phi/\omega(E),E)$ with $\omega(E)=2\pi/T(E)>0$ for all $E>E_0$. It follows easily that
\begin{gather}
    \label{HamSys}
    \omega(E)\frac{\partial X}{\partial \phi}=\partial_Y H(X,Y), \quad
    \omega(E)\frac{\partial Y}{\partial \phi}=-\partial_X H(X,Y), \quad  H(X(\phi,E),Y(\phi,E))\equiv E.
\end{gather}
Application of the averaging method~\cite{BM61,AKN06} to system \eqref{HamSys} yields the following asymptotic expansions as $E\to\infty$:
\begin{gather}
\label{XYw}
  X(\phi,E)=E^{\frac{1}{2h}}\sum_{j=0}^\infty E^{-\frac{j}{2h}} X_j(\phi), \quad
  Y(\phi,E)=E^{\frac{1}{2}}\sum_{j=0}^\infty E^{-\frac{j}{2h}} Y_j(\phi), \quad \omega(E)=E^{\frac{h-1}{2h}}\sum_{j=0}^\infty E^{-\frac{j}{2h}} \omega_j
\end{gather}
with $2\pi$-periodic coefficients $X_j(\phi)$, $Y_j(\phi)$ and $\omega_j={\hbox{\rm const}}$. Let $\xi(\alpha,J)$, $\eta(\alpha,J)$ be a family of $2\pi$-periodic solutions of the system \begin{gather*}
    \chi(J)\frac{\partial\xi}{\partial \alpha}=\eta, \quad \chi(J)\frac{\partial\eta}{\partial \alpha}=-\xi^{2h-1}, \quad \frac{\eta^2}{2}+\frac{\xi^{2h}}{2h}=J>0, \quad \chi(J)= \frac{2\pi}{  \kappa}J^{\frac{h-1}{2h}}.
\end{gather*}
Then, it is not hard to check that
\begin{gather}\label{omega}
    \omega_0 = \chi(1),\quad \omega_1=0,\quad  \omega_2=u_{2h-2} \big\langle\partial_J  (\xi^{2h-2} \chi^{-1}   )\big \rangle\big|_{J=1}+\frac{u^2_{2h-1} }{2} \big\langle\xi ^{4h-2}\partial_J^2\chi \big\rangle\big|_{J=1},
\end{gather}
and
\begin{equation}\label{XY}
\begin{aligned}
   \begin{pmatrix}  X_0 \\ Y_0 \end{pmatrix} =  & \begin{pmatrix}  \xi(\phi,1) \\ \eta(\phi,1) \end{pmatrix},\quad
    \begin{pmatrix}  X_1 \\ Y_1 \end{pmatrix} =   \alpha_1(\phi) \begin{pmatrix}  \partial_\alpha\xi(\phi,1) \\ \partial_\alpha\eta(\phi,1) \end{pmatrix} +J_1(\phi) \begin{pmatrix}  \partial_J\xi(\phi,1) \\ \partial_J\eta(\phi,1) \end{pmatrix},\\
     \begin{pmatrix}  X_2 \\ Y_2 \end{pmatrix} =  & \alpha_2(\phi) \begin{pmatrix}  \partial_\alpha\xi(\phi,1) \\ \partial_\alpha\eta(\phi,1) \end{pmatrix} +\big(J_2(\phi)+\partial_\alpha J_1(\phi) \alpha_1(\phi)\big) \begin{pmatrix}  \partial_J\xi(\phi,1) \\ \partial_J\eta(\phi,1) \end{pmatrix}\\
      &  +\frac{\alpha_1^2(\phi)}{2} \begin{pmatrix}  \partial_\alpha^2\xi(\phi,1) \\ \partial_\alpha^2\eta(\phi,1) \end{pmatrix} + \alpha_1(\phi) J_1(\alpha) \begin{pmatrix}  \partial_\alpha\partial_J\xi(\phi,1) \\ \partial_\alpha\partial_J\eta(\phi,1) \end{pmatrix}+\frac{J_1^2(\phi)}{2} \begin{pmatrix}  \partial_J^2\xi(\phi,1) \\ \partial_J^2\eta(\phi,1) \end{pmatrix},
      \end{aligned}
\end{equation}
where
\begin{equation*}
\begin{aligned}
    \alpha_1(\phi)=&   \alpha_1^0+\omega_0 u_{2h-1}\int\limits_0^\phi \partial_J \big(\xi^{2h-1} \chi^{-1}  \big)\big|_{J=1}\,d\varsigma, \quad
    J_1(\phi)=-u_{2h-1} \xi^{2h-1}\big|_{J=1}, \\
    \alpha_2(\phi)=& \alpha_2^0+  \chi''(1)  \frac{u^2_{2h-1}}{2} \int\limits_0^\phi \big(\xi^{4h-2}  -\big\langle \xi^{4h-2}  \big\rangle\big)\big|_{J=1} \,d\varsigma +\int\limits_0^\phi \alpha_1''\alpha_1-\big\langle \alpha_1''\alpha_1\big\rangle\,d\varsigma\\
                 & + \omega_0 u_{2h-2}\int\limits_0^\phi \Big(\partial_J  (\xi^{2h-2} \chi^{-1}   )  -\big\langle \partial_J  (\xi^{2h-2} \chi^{-1}   ) \big\rangle\Big)\Big|_{J=1} \,d\varsigma,\\
    J_2(\phi)=&-\left(u_{2h-2}  \xi  ^{2h-2}+u_{2h-1} J_1(\phi)\partial_J \xi^{2h-1} \right)\Big|_{J=1}.
    \end{aligned}
\end{equation*}
The parameters $\alpha_1^0$, $\alpha_2^0$ are chosen such that
\begin{gather*}
  \langle \alpha_i(\phi)\rangle:=\frac{1}{2\pi}\int\limits_{0}^{2\pi}\alpha_i(\varsigma)\,d\varsigma=0.
\end{gather*}

The functions $X(\phi,E)$, $Y(\phi,E)$ are used for rewriting system \eqref{FulSys} in the energy-angle variables. In particular, the change of variables
\begin{gather}
\label{exch1}
    x(t)=X\big(\phi(t),I(t)\big), \quad y(t)=Y\big(\phi(t),I(t)\big)
\end{gather}
transforms system \eqref{FulSys} into the form:
\begin{gather}
\label{IPS}
    \frac{dI}{dt}=t^{-\frac{a}{q}}F(I,\phi,t), \quad \frac{d\phi}{dt}=\omega(I)+t^{-\frac{a}{q}}G(I,\phi,t),
\end{gather}
where
\begin{eqnarray*}
    F(I,\phi,t)&\equiv &f\big(X(\phi,I),Y(\phi,I),S(t),t\big) U'\big(X(\phi,I)\big)+g\big(X(\phi,I),Y(\phi,I),S(t),t\big) Y(\phi,I), \\
    G(I,\phi,t)&\equiv&\omega(I)\Big(f\big(X(\phi,I),Y(\phi,I),S(t),t\big)\partial_E Y(\phi,I)-g\big(X(\phi,I),Y(\phi,I),S(t),t\big)\partial_E X(\phi,I)\Big).
\end{eqnarray*}
Moreover, it follows from \eqref{fg} that
\begin{gather}\label{FG}
    F(I,\phi,t) =\sum_{k=0}^\infty t^{-\frac kq} F_k(I,\phi,S(t)), \quad
    G(I,\phi,t)=\sum_{k=0}^\infty t^{-\frac kq} G_k(I,\phi,S(t)), \quad t\to\infty,
\end{gather}
where the coefficients
\begin{eqnarray*}
 F_k(I,\phi,S)&\equiv &f_k\big(X(\phi,I),Y(\phi,I),S\big) U'\big(X(\phi,I)\big)+g_k\big(X(\phi,I),Y(\phi,I),S\big) Y(\phi,I), \\
 G_k(I,\phi,S)&\equiv&\omega(I)\Big(f_k\big(X(\phi,I),Y(\phi,I),S\big)\partial_E Y(\phi,I)-g_k\big(X(\phi,I),Y(\phi,I),S\big)\partial_E X(\phi,I)\Big)
\end{eqnarray*}
are $2\pi$-periodic functions with respect to $\phi$ and $S$.
Define
\begin{align*}
  &\tilde F_k(I,\phi,S) \equiv  I^{-1-\frac{p-1+(l-1)(h-1)}{2h}}F_k(I,\phi,S), &\quad& \tilde F(I,\phi,t) \equiv  I^{-1-\frac{p-1+(l-1)(h-1)}{2h}}F(I,\phi,t),\\
   & \tilde G_k(I,\phi,S) \equiv  I^{-\frac{p-1+(l-1)(h-1)}{2h}}G_k(I,\phi,S), &\quad&
    \tilde G(I,\phi,t) \equiv  I^{-\frac{p-1+(l-1)(h-1)}{2h}}G(I,\phi,t).
\end{align*}
Then, taking into account \eqref{XYw}, we have the following asymptotics:
\begin{gather}
\label{FGk}
\tilde F_k(I,\phi,S)=\sum_{d=0}^\infty \tilde F_{k,d}(\phi,S) I^{-\frac{d}{2h}}, \quad
\tilde G_k(I,\phi,S)=\sum_{d=0}^\infty \tilde G_{k,d}(\phi,S) I^{-\frac{d}{2h}}, \quad I\to\infty,
\end{gather}
with $2\pi$-periodic coefficients:
\begin{eqnarray*}
  \tilde F_{k,d}(\phi,S)&=&\sum_{(i,j,i_1,j_1,i_2,j_2)\in\mathcal X_d } \delta_{j_2,0}\big(X_0(\phi)\big)^{p-l-i} \big(Y_0(\phi)\big)^{l-1-j}\\
  &&\times\Big( (2h-i_2)u_{2h-i_2}  A_{k,p-l+1-i,l-1-j}(S)\tilde X_{p-l-i+2h-i_2,i_1}(\phi)\tilde Y_{l-1-j,j_1}(\phi)\big(X_0(\phi)\big)^{2h-i_2} \\
  &&+ \delta_{i_2,0} B_{k,p-l-i,l-j}(S)\tilde X_{p-l-i,i_1}(\phi)\tilde Y_{l+1-j,j_1}(\phi)\big(Y_0(\phi)\big)^2\Big),\\
  \tilde G_{k,d}(\phi,S)&=&\sum_{(i,j,i_1,j_1,i_2,j_2)\in\mathcal X_d} \frac{ \omega_{j_2}}{2h}\big(X_0(\phi)\big)^{p-l-i} \big(Y_0(\phi)\big)^{l-1-j}\\
  &&\times\Big( (h-i_2) A_{k,p-l+1-i,l-1-j}(S)\tilde X_{p-l+1-i,i_1}(\phi)\tilde Y_{l-1-j,j_1}(\phi) X_0(\phi)Y_{i_2}(\phi)  \\
  &&-  (1-i_2) B_{k,p-l-i,l-j}(S)\tilde X_{p-l-i,i_1}(\phi)\tilde Y_{l-j,j_1}(\phi) Y_0(\phi)X_{i_2}(\phi)  \Big),
\end{eqnarray*}
where $\delta_{i,0}$ is the Kronecker delta,
\begin{align*}
  \mathcal X_d=\big\{(i,j,i_1,j_1,i_2,j_2)\in\mathbb Z^6:\,
  0\leq i+j\leq p,\, 0 \leq j\leq l,\,i_1,j_1,i_2,j_2\geq 0, \\
  i+j+(h-1)j+i_1+j_1+i_2+j_2=d\big\}.
\end{align*}
The functions $\tilde X_{n,i}(\phi)$, $ \tilde Y_{n,i}(\phi)$ denote the coefficients of the asymptotic expansions:
\begin{gather*}
 I^{-\frac{n}{2h}}\left(\frac{X(\varphi,I)}{ X_0(\varphi)}\right)^n=\sum_{i=0}^\infty \tilde X_{n,i}(\phi) I^{-\frac{i}{2h}}, \quad
I^{-\frac{n}{2}}\left(\frac{Y(\varphi,I)}{Y_0(\varphi)}\right)^n=\sum_{i=0}^\infty \tilde Y_{n,i}(\phi) I^{-\frac{i}{2h}}, \quad I\to\infty.\end{gather*}
For example, $\tilde X_{0,0} =1$, $\tilde X_{0,i}=0$, $\tilde X_{n,0}=1$, $\tilde X_{n,1}=n X_1 /X_0 $, $\tilde X_{n,2}=n X_2 /X_0 +n(n-1) (X_1 /X_0)^2/2$ for all $n,i\neq 0$. It is assumed that $u_{2h}=1/(2h)$, $u_{2h-i_2}= 0$ for all $i_2\geq 2h$ and $A_{k,i,-1}(S)\equiv 0$ for all $k,i\geq 0$.

From \eqref{HamSys} it follows that
\begin{gather*}
\det \frac{\partial (X,Y)}{\partial (\phi,E)} =  \begin{vmatrix}
        \partial_\phi X & \partial_E X\\
        \partial_\phi Y& \partial_E Y
    \end{vmatrix} =  \frac{1}{\omega(E)}> 0  \quad \forall \, E>E_0.
\end{gather*}
Hence, the transformation \eqref{exch1} is invertible for all $I>E_0$ and $\phi\in \mathbb R$. Define $\mathcal D({E_0})=\{(x,y)\in\mathbb R^2: H(x,y)>E_0\}$. Then, we have the following.
\begin{Lem}\label{Lem1}
For all $(x,y)\in \mathcal D({E_0})$ and $t>0$ system \eqref{FulSys} can be transformed into \eqref{IPS} by the transformation \eqref{exch1}.
\end{Lem}

\subsection{Amplitude and phase difference}
It follows from \eqref{XYw} that $\omega'(E)> 0$ for all $E\geq E_1$ with some $E_1\geq E_0$. Hence, there exists $t_1>0$ such that the equation
\begin{gather}\label{omeq}
  \omega\big(I_\ast(t)\big)\equiv\varkappa^{-1} S'(t)
\end{gather}
has a smooth solution $I_\ast(t)\geq E_1$ defined for all $t\geq t_1$. Define
\begin{gather}\label{lam}
  z(t) \equiv c_\varkappa  t^{-\frac{b}{(h-1)q}} \big( I_\ast(t)\big)^{\frac{1}{2h}}, \quad
  c_\varkappa=\left( \frac{\omega_0  \varkappa}{\vartheta}\right)^{\frac{1}{h-1}}, \quad  \vartheta= s \Big(1+\frac{b}{q}\Big)>0.
\end{gather}
Then, it can be easily seen that the function $z(t)$ has power-law asymptotics:
\begin{gather}\label{zt}
    z(t)=\sum_{k=0}^\infty z_k t^{-\frac{ k b}{(h-1)q}}, \quad t\to\infty,
\end{gather}
with constant coefficients $z_k$. In particular, substituting \eqref{zt} into \eqref{omeq}, we obtain
\begin{gather*}
  z_0=1,\quad  z_1=0, \quad z_2=- \frac{\omega_2 c_\varkappa^{2} }{ (h-1)\omega_0}, \quad z_3=- \frac{\omega_3 c_\varkappa^{3} }{ (h-1)\omega_0}.
\end{gather*}

The solution  $I_\ast(t)$ of equation \eqref{omeq} is used in the following change of variables in system \eqref{IPS}:
\begin{gather}
 \label{exch2} I(t)= I_\ast(t)\big(1+t^{-\mu} r(\tau)\big)^{2h}, \quad \phi(t)=\theta(\tau)+\varkappa^{-1} S(t), \quad \tau=\frac{t^{\nu}}{\nu}
\end{gather}
with some constants  $\mu>0$, $\nu>0$, $\varkappa\in\mathbb Z_+$. Note that for all $r_0>0$ there exists $t_2=\max\{t_1,(2r_0)^{1/\mu}\}$ such that $t^{-\mu}|r|<1/2$ for all $|r|\leq r_0$ and $t\geq t_2$. In this case, the mapping $(I,\phi,t)\mapsto (r,\theta,\tau)$ is invertible for all  $|r|\leq r_0$, $\theta\in\mathbb R$ and $\tau\geq \tau_2=t_2^\nu/\nu$.

We take
\begin{gather}\label{munu}
   \mu=\frac{1}{2}\left(\frac{b}{q}-\sigma \right), \quad \nu=1+2\mu+\sigma,\quad
    M=2\mu q  (h-1), \quad    N=2\nu q(h-1).
\end{gather}
It follows from \eqref{rc} that
\begin{gather*}
  \mu>0, \quad \nu-2\mu\geq 0, \quad
  M,N\in\mathbb Z, \quad 1\leq M\leq (h-1)(b+q), \quad 0\leq N-2M<2(h-1)(b+q).
\end{gather*}
Hence, substituting \eqref{exch2} into \eqref{IPS} yields the following asymptotically autonomous system:
\begin{gather}\label{RPsi}
    \frac{dr}{d\tau}=\tau^{-\frac{M}{N}}\mathcal A(r,\theta,\tau), \quad
    \frac{d\theta}{d\tau}= \tau^{-\frac{M}{N}}\mathcal B(r,\theta,\tau),
\end{gather}
with the right-hand sides
\begin{gather*}
    \mathcal A(r,\theta,\tau)\equiv \mathcal F(r,\theta,\tau)+ \tau^{-\frac{N-2M}{N}}\mathcal P(r,\tau)+\tau^{-\frac{N-M}{N}}\frac{M}{N} r, \quad
    \mathcal B(r,\theta,\tau)\equiv\mathcal Q(r,\tau)r+ \tau^{-\frac{M}{N}}\mathcal G(r,\theta,\tau),
\end{gather*}
where
\begin{eqnarray*}
    \mathcal F(r,\theta,\tau)&\equiv & \tilde F \Big(I_\ast\big((\nu \tau)^{\frac 1 \nu}\big) \big(1+(\nu \tau)^{-\frac{\mu}{\nu}} r \big)^{2h},\theta+\varkappa^{-1}\zeta(\tau),\zeta(\tau),(\nu \tau)^{\frac 1 \nu}\Big) \frac{c_\varkappa\nu^{-\frac{\mu}{\nu}}}{  2 h z\big((\nu \tau)^{\frac 1 \nu}\big)} \\
    & & \times\Big(c_\varkappa^{-1}z\big((\nu \tau)^{\frac 1 \nu}\big)\big(1+(\nu \tau)^{-\frac{\mu}{\nu}} r \big)\Big)^{p+(l-1)(h-1)}, \\
    \mathcal P(r,\tau) &\equiv& -  \big(1+(\nu \tau)^{-\frac{\mu}{\nu}} r \big)\nu^{-\frac{\nu-\mu}{\nu}} \left(\frac{tI_\ast'(t)}{2h I_\ast(t)}\right)\Big|_{t=(\nu \tau)^{1/\nu}},\\
    \mathcal Q(r,\tau) &\equiv& r^{-1}\left[\omega\Big(I_\ast\big((\nu \tau)^{\frac 1 \nu}\big)\big(1+(\nu \tau)^{-\frac{\mu}{\nu}} r \big)^{2h}\Big)-\omega\Big(I_\ast\big((\nu \tau)^{\frac 1 \nu}\big)\Big)\right] \tau ^{\frac{\mu-\nu+1}{\nu}}\nu^{\frac{1-\nu}{\nu}},\\
    \mathcal G(r,\theta,\tau) &\equiv& \tilde G \Big(I_\ast\big((\nu \tau)^{\frac 1 \nu}\big) \big(1+(\nu \tau)^{-\frac{\mu}{\nu}} r \big)^{2h},\theta+\varkappa^{-1}\zeta(\tau),\zeta(\tau),(\nu \tau)^{\frac 1 \nu}\Big)\nu^{-\frac{2\mu}{\nu}}\\
    &&\times \Big(c_\varkappa^{-1}z\big((\nu \tau)^{\frac 1 \nu}\big)\big(1+(\nu \tau)^{-\frac{\mu}{\nu}} r \big)\Big)^{p-1+(l-1)(h-1)},\\
    \zeta(\tau)&=&S\big((\nu \tau)^{\frac 1 \nu}\big).
\end{eqnarray*}
Note that the condition \eqref{rc} guarantees that the transformed system \eqref{RPsi} is asymptotically autonomous and not trivial in the leading asymptotic terms.

Combining \eqref{FG}, \eqref{FGk} and  \eqref{lam}, we obtain the asymptotic estimates as $\tau\to\infty$:
\begin{gather*}
    \mathcal F   = \sum_{K=0}^\infty \mathcal F_K(r,\theta,\zeta)\tau^{- \frac{K}{N} },\ \
    \mathcal P   = \sum_{K=0}^\infty \mathcal P_K(r)\tau^{- \frac{K}{N} },\ \
    \mathcal Q  = \sum_{K=0}^\infty \mathcal Q_K(r)\tau^{- \frac{K}{N} },\ \
    \mathcal G  = \sum_{K=0}^\infty \mathcal G_K(r,\theta,\zeta)\tau^{- \frac{K}{N} }
\end{gather*}
for all $|r|\leq r_0$, $\theta\in\mathbb R$, where
\begin{align*}
   \mathcal F_K(r,\theta,\zeta) &\equiv\nu^{-\frac{K+M}{N}} \sum\limits_{(k,d,i,\ell)\in \mathcal Y_K} \hat F_{k,d,i}^{\ell}(\theta,\zeta)r^{\ell},&\quad
   \mathcal P_K(r) &\equiv \nu^{-\frac{K+N-M}{N}}\sum\limits_{(k,d,i,\ell)\in \mathcal Y_K, \ell\in\{0,1\}} \hat P_{k,d,i}^{\ell} r^{\ell}, \\
   \mathcal Q_K(r) &\equiv  \nu^{-\frac{K+M}{N}}\sum\limits_{(k,d,i,\ell)\in \mathcal Y_K}  \hat Q_{k,d,i}^{\ell} r^{\ell},&\quad
    \mathcal G_K(r,\theta,\zeta)&\equiv\nu^{-\frac{K+2M}{N}} \sum\limits_{(k,d,i,\ell)\in \mathcal Y_K} \hat G_{k,d,i}^{\ell}(\theta,\zeta)r^{\ell},
\end{align*}
and
\begin{eqnarray*}
    \hat F_{k,d,i}^{\ell}(\theta,\zeta)&\equiv& \frac{1}{2h} \tilde F_{k,d}(\theta+\varkappa^{-1}\zeta,\zeta)C_{p+(l-1)(h-1)-d,\ell}Z_{p-1+(l-1)(h-1)-d,i} c_\varkappa^{- ( p-1+(l-1)(h-1)-d)} ,\\
    \hat G_{k,d,i}^{\ell}(\theta,\zeta)&\equiv&   \tilde G_{k,d}(\theta+\varkappa^{-1}\zeta,\zeta)C_{p-1+(l-1)(h-1)-d,\ell}Z_{p-1+(l-1)(h-1)-d,i} c_\varkappa^{- ( p-1+(l-1)(h-1)-d)},\\
    \hat Q_{k,d,i}^{\ell}&=& \omega_{i} \delta_{k,0} C_{h-1-i,\ell+1} Z_{h-1-i,d}c_\varkappa^{- (h-1-i)},\\
     \hat P_{k,d,i}^{\ell}&=&-\delta_{k,0}\frac{(2h-i)b}{2h(h-1)q}C_{1,\ell} Z_{2h,d} Z_{-2h,i}.
\end{eqnarray*}
Here $\mathcal Y_K=\{(k,d,i,\ell)\in\mathbb Z^4: k\geq 0, \, d\geq 0,\, i\geq 0,\, \ell\geq 0,\,   2(h-1)k+2b(d+i)+M\ell=K\}$, and the parameters $C_{n,i}$, $ Z_{n,i}$ denote the coefficients of the asymptotic expansions
\begin{gather*}
(1+t^{-\mu})^n=\sum_{i=0}^\infty C_{n,i} t^{-\mu i}, \quad
\big(z(t)\big)^n=\sum_{i=0}^\infty Z_{n,i} t^{-\frac{ b i}{(h-1)q}}, \quad t\to\infty.
\end{gather*}
In particular, $C_{0,0}=Z_{0,0}=1$, $C_{0,i}=Z_{0,i}=0$, $C_{n,0}=1$, $C_{n,1}=n$, $C_{n,2}=n(n-1)/2$, $Z_{n,0}=1$,$Z_{n,1}=0$, $Z_{n,2}=n z_2$ for all $n,i\neq 0$.

Thus, the right-hand sides of system \eqref{RPsi} have the following asymptotics:
\begin{gather*}
        \mathcal A(r,\theta,\tau)  = \sum_{K=0}^\infty \tau^{- \frac{K}{N} }\mathcal A_K(r,\theta,\zeta(\tau)),
     \quad
         \mathcal B(r,\theta,\tau)  = \sum_{K=0}^\infty \tau^{- \frac{K}{N} }\mathcal B_K(r,\theta,\zeta(\tau))
\end{gather*}
as $\tau\to\infty$ uniformly for all $|r|\leq r_0$, $\theta\in\mathbb R$, with
\begin{gather}\label{AKBK}
\begin{split}
  &\mathcal A_K(r,\theta,\zeta)\equiv \mathcal F_K(r,\theta,\zeta)+\mathcal P_{K+2M-N}(r)+\delta_{K,N-M}\frac{M}{N}r,\\
   &\mathcal B_K(r,\theta,\zeta)\equiv\mathcal Q_K(r) r+\mathcal G_{K-M}(r,\theta,\zeta).
    \end{split}
\end{gather}
It is assumed that $\mathcal P_{i}\equiv \mathcal G_{i}\equiv 0$ if $i<0$.  It follows easily that $\mathcal A_K(r,\theta,\zeta)$ and  $\mathcal B_K(r,\theta,\zeta)$ are $2\pi$-periodic with respect to $\theta$ and $2\pi \varkappa$-periodic with respect to $\zeta$. In particular,
\begin{align*}
    \begin{pmatrix}  \mathcal A_K\\ \mathcal B_K\end{pmatrix}
        & \equiv
    \begin{pmatrix}\mathcal A_K^0(\theta,\zeta) \\  \mathcal Q_K^0 r \end{pmatrix}, &\quad &K\in[0,M),\\
     \begin{pmatrix}  \mathcal A_K\\ \mathcal B_K\end{pmatrix}
        & \equiv
     \begin{pmatrix}\mathcal A_K^1(\theta,\zeta)r+\mathcal A_K^0(\theta,\zeta) \\ \mathcal Q_K^1 r^2+\mathcal Q_K^0 r +\mathcal G_{K-M}^0(\theta,\zeta)  \end{pmatrix}, &\quad&   K\in[M,2M),\\
    \begin{pmatrix}  \mathcal A_K\\ \mathcal B_K\end{pmatrix} & \equiv
    \begin{pmatrix} \mathcal A_K^2(\theta,\zeta)r^2+\mathcal A_K^1(\theta,\zeta)r+ \mathcal A_K^0(\theta,\zeta) \\ \mathcal Q_K^2 r^3 + \mathcal Q_K^1 r^2+\big(\mathcal Q_K^0  + \mathcal G_{K-M}^1(\theta,\zeta)\big)r   + \mathcal G_{K-M}^0(\theta,\zeta)  \end{pmatrix}, &\quad& K\in[2M,3M),
\end{align*}
where
\begin{equation*}
\begin{aligned}
  \mathcal A_K^\ell(\theta,\zeta)\equiv&\nu^{-\frac{K+M}{N}} \sum\limits_{(k,d,i,\ell)\in \mathcal Y_K} \hat F_{k,d,i}^{\ell}(\theta,\zeta)+\nu^{-\frac{K+2M-N}{N}} \sum\limits_{(k,d,i,\ell)\in \mathcal Y_{K+N-M}} \hat P_{k,d,i}^{\ell}+\delta_{\ell,1}\delta_{K,N-M}\frac{M}{N},\\
  \mathcal G_{K}^\ell(\theta,\zeta)\equiv&\nu^{-\frac{K+2M}{N}} \sum\limits_{(k,d,i,\ell)\in \mathcal Y_{K}} \hat G_{k,d,i}^{\ell}(\theta,\zeta),\qquad
  \mathcal Q_K^\ell \equiv \nu^{-\frac{K+M}{N}} \sum\limits_{(k,d,i,\ell)\in \mathcal Y_K} \hat Q_{k,d,i}^{\ell}.
\end{aligned}
\end{equation*}

\begin{Lem}\label{Lem2}
Let assumption \eqref{rc} hold. Then for all $E\geq E_1$, $\phi\in\mathbb R$ and $t\geq t_1$ system \eqref{IPS} can be transformed into \eqref{RPsi} by the transformation \eqref{exch2}.
\end{Lem}

\subsection{Averaging}
Note that $d\zeta/d\tau=\vartheta$. Hence, $\zeta(\tau)$ changes rapidly in comparison to potential variations of $r(\tau)$ and $\theta(\tau)$ for large values of $\tau$. Further simplification of the system is associated with the averaging of the equations over $\zeta$. This technique is usually used in perturbation theory (see, for example,~\cite{AKN06,AN84,BDP01,DM10}).

Consider the following near-identity transformation:
\begin{gather}\label{RnPsim}
  R_n(r,\theta,\tau)=r+\sum_{K=0}^{n} \tau^{-\frac{M+K}{N}} \rho_K(r,\theta,\zeta(\tau)), \quad   \Psi_n(r,\theta,\tau)=\theta+\sum_{K=0}^{n} \tau^{-\frac{M+K}{N}} \psi_k(r,\theta,\zeta(\tau))
\end{gather}
with some integer $n\geq 0$. The coefficients $\rho_K(r,\theta,\zeta)$, $\psi_K(r,\theta,\zeta)$ are sought in such a way that the right-hand sides of the transformed equations in the variables $R(\tau)\equiv R_n(r(\tau),\theta(\tau),\tau)$ and $\Psi(\tau)\equiv \Psi_n(r(\tau),\theta(\tau),\tau)$ do not depend explicitly on $\zeta$, at least in the first terms of the asymptotics:
\begin{gather}\label{RPsiAv}
\begin{split}
    \frac{dR}{d\tau}=\sum_{K=0}^{n} \tau^{-\frac{M+K}{N}}\Lambda_K(R,\Psi)+\widetilde{\Lambda}_n(R,\Psi,\tau),\quad  \frac{d\Psi}{d\tau}=\sum_{K=0}^{n} \tau^{-\frac{M+K}{N}}\Omega_K(R,\Psi)+\widetilde{\Omega}_n(R,\Psi,\tau),
\end{split}
\end{gather}
and the remainders $\widetilde{\Lambda}_n(R,\Psi,\tau)$, $\widetilde{\Omega}_n(R,\Psi,\tau)$ satisfy the estimates $ \widetilde{\Lambda}_n(R,\Psi,\tau)=\mathcal O(\tau^{-(M+n+1)/N})$, $\widetilde{\Omega}_n(R,\Psi,\tau)=\mathcal O(\tau^{-(M+n+1)/N})$ as $\tau\to\infty$. Calculating the total derivative of $R_n(r,\theta,\tau)$ and $\Psi_n(r,\theta,\tau)$ with respect to $\tau$ along the trajectories of system \eqref{RPsi} yields
\begin{eqnarray}
\nonumber    \frac{d}{d\tau}\begin{pmatrix} R_n \\ \Psi_n\end{pmatrix}\bigg |_\eqref{RPsi}
        &=&
            \Big(\tau^{-\frac{M}{N}}\big(\mathcal A\partial_r+\mathcal B\partial_\theta\big)+\partial_\tau\Big)\begin{pmatrix} R_n \\ \Psi_n\end{pmatrix}\\
 \label{TD}       &=& \sum_{K=0}^\infty \tau^{-\frac{M+K}{N}} \left[
     \vartheta \partial_\zeta\begin{pmatrix} \rho_K\\ \psi_K\end{pmatrix}+\begin{pmatrix} \mathcal A_K \\ \mathcal B_K\end{pmatrix}+\frac{N-K}{N}\begin{pmatrix} \rho_{K-M-N}\\ \psi_{K-M-N}\end{pmatrix}\right]\\
\nonumber       &&+\sum_{K=0}^\infty\tau^{-\frac{2M+K}{N}}\sum_{i+j=K}
        \big(\mathcal A_i \partial_r + \mathcal B_i \partial_\theta\big)\begin{pmatrix} \rho_j\\ \psi_j\end{pmatrix},
\end{eqnarray}
where it is assumed that $\rho_i\equiv \psi_i\equiv \mathcal A_j\equiv \mathcal B_j\equiv 0$ if $i,j< 0$ or $i>n$.
Matching \eqref{TD} with \eqref{RPsiAv} gives the following chain of differential equations for determining $\rho_K$ and $\psi_K$:
\begin{gather}
\label{rpsys}
  \vartheta\partial_\zeta
    \begin{pmatrix} \rho_K\\ \psi_K\end{pmatrix}
        =
    \begin{pmatrix}
        \Lambda_K(r,\theta)-\mathcal A_K(r,\theta,\zeta)-\mathfrak A_K(r,\theta,\zeta) \\
        \Omega_K(r,\theta)-\mathcal B_K(r,\theta,\zeta)-\mathfrak B_K(r,\theta,\zeta)
    \end{pmatrix},
  \quad K\geq 0,
\end{gather}
where the functions $\mathfrak A_K$, $\mathfrak B_K$ are expressed through $\{\rho_i,\psi_i,\Lambda_i,\Omega_i\}_{j=0}^{K-M}$. In particular,
\begin{align*}
    \begin{pmatrix} \mathfrak A_K\\ \mathfrak B_K\end{pmatrix}  \equiv &
    \begin{pmatrix} 0\\ 0\end{pmatrix}, &\quad & K\in[0,M),\\
    \begin{pmatrix} \mathfrak A_K\\ \mathfrak B_K\end{pmatrix}  \equiv &
    \sum_{i+j=K-M}\left[\big(\mathcal A_i \partial_r + \mathcal B_i \partial_\theta\big)\begin{pmatrix} \rho_j\\ \psi_j\end{pmatrix}-(\rho_i\partial_r+\psi_i\partial_\theta)\begin{pmatrix} \Lambda_j\\ \Omega_j\end{pmatrix} \right], &\quad & K\in[M,2M),\\
    \begin{pmatrix} \mathfrak A_{K}\\ \mathfrak B_{K}\end{pmatrix}  \equiv &
    \sum_{i+j=K-M}\left[\big(\mathcal A_i \partial_r + \mathcal B_i \partial_\theta\big)\begin{pmatrix} \rho_j\\ \psi_j\end{pmatrix}-(\rho_i\partial_r+\psi_i\partial_\theta)\begin{pmatrix} \Lambda_j\\ \Omega_j\end{pmatrix} \right] & & \\
     & - \frac{1}{2} \sum_{i+j+k=K-2M}\Big( \rho_i\rho_j  \partial_r^2 + (\rho_i \psi_j +\rho_j \psi_i) \partial_r \partial_\theta + \psi_i\psi_j \partial_\theta^2 \Big) \begin{pmatrix} \Lambda_k \\ \Omega_k\end{pmatrix},&\quad  & K\in[2M,3M),
\end{align*}
etc. Define
\begin{gather}
\label{LOm}
    \Lambda_K(r,\theta)\equiv \langle \mathcal A_K(r,\theta,\zeta)+\mathfrak A_K(r,\theta,\zeta)\rangle_{\varkappa\zeta},\quad
    \Omega_K(r,\theta)\equiv \langle \mathcal B_K(r,\theta,\zeta)+\mathfrak B_K(r,\theta,\zeta)\rangle_{\varkappa\zeta},
\end{gather}
where
\begin{gather*}
  \langle \mathfrak C (r,\theta,\zeta) \rangle_{\varkappa\zeta} \equiv  \frac{1}{2\pi}\int\limits_0^{2\pi}\mathfrak C(r,\theta,\varkappa\varsigma)\, d\varsigma \equiv \frac{1}{2\pi\varkappa}\int\limits_0^{2\pi\varkappa}\mathfrak C(r,\theta,\varsigma)\, d\varsigma.
\end{gather*}
Then, for all $K\geq 0$ the right-hand sides of system \eqref{rpsys} are $2\pi\varkappa$-periodic with respect to $\zeta$ with zero average. Integrating \eqref{rpsys} yileds
\begin{gather}\label{rhopsi}
  \begin{pmatrix} \rho_K(r,\theta,\zeta)\\ \psi_K(r,\theta,\zeta)\end{pmatrix}= -\frac{1}{\vartheta}\int\limits_0^\zeta   \begin{pmatrix}
        \{\mathcal A_K(r,\theta,\zeta)+\mathfrak A_K(r,\theta,\zeta)\}_{\varkappa\zeta} \\
        \{\mathcal B_K(r,\theta,\zeta)+\mathfrak B_K(r,\theta,\zeta)\}_{\varkappa\zeta}
    \end{pmatrix}\,d\zeta+
    \begin{pmatrix} \hat{\rho}_K(r,\theta)\\ \hat{\psi}_K(r,\theta)\end{pmatrix},
\end{gather}
where $\{\mathfrak C\}_{\varkappa \zeta}:=\mathfrak C-\langle \mathfrak C\rangle_{\varkappa \zeta}$, and the functions $\hat{\rho}_K(r,\theta)$, $\hat{\psi}_K(r,\theta)$ are chosen such that $\langle \mathfrak \rho_K\rangle_{\varkappa \zeta}\equiv \langle \mathfrak \psi_K\rangle_{\varkappa \zeta}\equiv 0$. Thus, the functions $\rho_K(r,\theta,\zeta)$, $\psi_K(r,\theta,\zeta)$ are smooth and periodic with respect to $\theta$ and $\zeta$.

From \eqref{AKBK}, \eqref{rpsys} and \eqref{rhopsi} it follows that
\begin{align*}
     \begin{pmatrix} \Lambda_K\\ \Omega_K\end{pmatrix}
   & \equiv
        \begin{pmatrix} \Lambda_K^0(\theta)\\ \mathcal Q_K^0 r\end{pmatrix},
        \quad
        \begin{pmatrix} \rho_K\\  \psi_K\end{pmatrix}
        \equiv
        \begin{pmatrix} \rho_K^0(\theta,\zeta)\\  0\end{pmatrix},&
        \quad
        K&\in[0,M),
    \\
     \begin{pmatrix} \Lambda_K \\ \Omega_K \end{pmatrix}
     &\equiv
        \begin{pmatrix}  \Lambda_K^1(\theta)r+\Lambda_K^0(\theta)\\ \mathcal Q_K^1 r^2+\mathcal Q_K^0 r+\Omega_K^0(\theta)\end{pmatrix},
        \quad
        \begin{pmatrix} \rho_K \\  \psi_K \end{pmatrix}
        \equiv
        \begin{pmatrix}\rho_K^1(\theta,\zeta)r+ \rho_K^0(\theta,\zeta)\\  \psi_K^0(\theta,\zeta)\end{pmatrix},&
        \quad
        K&\in[M,2M),
    \\
    \begin{pmatrix} \Lambda_K \\ \Omega_K \end{pmatrix}
    &\equiv
        \begin{pmatrix}  \Lambda_K^2(\theta)r^2+\Lambda_K^1(\theta)r+\Lambda_K^0(\theta)\\ \mathcal Q_K^2 r^3+\mathcal Q_K^1 r^2+(\mathcal Q_K^0+\Omega_K^1(\theta)) r+\Omega_K^0(\theta)\end{pmatrix},&
        \quad
        K&\in[2M,3M),
\end{align*}
where $\Lambda_K^{0,1}\equiv \langle \mathcal A_K^{0,1}(\theta,\zeta)\rangle_{\varkappa\zeta}$, $\Omega_K^{0,1}\equiv \langle \mathcal G_{K-M}^{0,1}(\theta,\zeta)\rangle_{\varkappa\zeta}$,
\begin{equation*}
\begin{aligned}
  \Lambda_K^2\equiv& \langle  \mathcal A_K^{2}(\theta,\zeta) \rangle_{\varkappa\zeta}+\sum_{i=0}^{K-2M}\big\langle\mathcal A_i^{0}(\theta,\zeta)\rho_{K-M-i}^1(\theta,\zeta) \big\rangle_{\varkappa\zeta}+
  \sum_{i=0}^{K-2M-1}\big\langle \mathcal G_{K-2M-i}^0(\theta,\zeta)\partial_\theta\rho_i^0(\theta,\zeta)   \big\rangle_{\varkappa\zeta},\\
   \rho_K^0\equiv&-\frac{1}{\vartheta}\left\{\int\limits_0^\zeta\{\mathcal A_K^0(\theta,\zeta)\}_{\varkappa\zeta}\,d\zeta \right\}_{\varkappa\zeta}, \quad  \psi_K^0\equiv -\frac{1}{\vartheta}\left\{\int\limits_0^\zeta\Big\{\mathcal \mathcal \mathcal G_{K-M}^0(\theta,\zeta)-\sum_{i+j=K-M} \mathcal Q_{i}^0 \rho_j^0\Big\}_{\varkappa\zeta}\,d\zeta \right\}_{\varkappa\zeta},\\ \rho_K^1\equiv&-\frac{1}{\vartheta}\left\{\int\limits_0^\zeta\Big\{\mathcal A_K^1(\theta,\zeta) +\sum_{i+j=K-M} \mathcal Q_{i}^0 \partial_\theta \rho_j^0(\theta,\zeta) \Big\}_{\varkappa\zeta}\,d\zeta \right\}_{\varkappa\zeta}.
\end{aligned}
\end{equation*}
In this case, it can easily be checked that
\begin{eqnarray*}
    \widetilde{\Lambda}_n(R,\Psi,\tau) &\equiv& \frac{d}{d\tau}R_n(r,\theta,\tau)\Big|_\eqref{rpsys} - \sum_{K=0}^n \tau^{-\frac{M+K}{N}}\Lambda_K(R,\Psi)=\mathcal O\Big(\tau^{-\frac{M+n+1}{N}}\Big),
    \\
    \widetilde{\Omega}_n(R,\Psi,\tau) &\equiv& \frac{d}{d\tau}\Psi_n(r,\theta,\tau)\Big|_\eqref{rpsys} - \sum_{K=0}^n \tau^{-\frac{M+K}{N}}\Omega_K(R,\Psi)=\mathcal O\Big(\tau^{-\frac{M+n+1}{N}}\Big)
\end{eqnarray*}
as $\tau\to\infty$ uniformly for all $|R|\leq d_0$ and $\Psi\in\mathbb R$ with some $d_0={\hbox{\rm const}}>0$.

It follows from \eqref{RnPsim} that for all $r_0>0$ and $\varepsilon\in(0,r_0)$ there exists $\tau_0\geq \tau_2$ such that
\begin{align*}
    &|R_n(r,\theta,\tau)-r|\leq \varepsilon,& \quad& |\partial_r R_n(r,\theta,\tau)-1|\leq \varepsilon, &\quad &|\partial_\theta R_n(r,\theta,\tau)|\leq \varepsilon,  \\
    &|\Psi_n(r,\theta,\tau)-\theta|\leq \varepsilon, &\quad &|\partial_r \Psi_n(r,\theta,\tau)|\leq \varepsilon,& \quad &|\partial_\theta \Psi_n(r,\theta,\tau)-1|\leq \varepsilon
\end{align*}
for all $|r|\leq r_0$, $\theta\in\mathbb R$ and $\tau\geq \tau_0$. Hence, the mapping $(r,\theta)\mapsto (R,\Psi)$ is invertible for all  $|R|\leq d_0$, $\Psi\in\mathbb R$ and $\tau\geq \tau_0$ with $d_0=r_0-\varepsilon>0$.

\begin{Lem}\label{Lem3}
For all $r_0>0$ there exists $\tau_0>0$ such that for all $|r|\leq r_0$, $\theta\in\mathbb R$ and $\tau\geq \tau_0$ system \eqref{RPsi} can be transformed into \eqref{RPsiAv} by the transformation \eqref{RnPsim}.
\end{Lem}

Thus, combining Lemmas~\ref{Lem1},~\ref{Lem2} and \ref{Lem3}, we obtain the following.
\begin{Th}\label{Th1}
Let assumptions \eqref{fg}, \eqref{rc} hold.  Then there exists $t_0>0$ such that for all $(x,y)\in\mathcal D(E_1)$ and $t>t_0$ system \eqref{FulSys} can be transformed into \eqref{RPsiAv} by the transformations \eqref{exch1}, \eqref{exch2} and \eqref{RnPsim}.
\end{Th}

\section{Asymptotic regimes}
\label{sec3}
Let $0\leq L\leq \min\{N-M,2M-1\}$ be a whole number such that
\begin{gather}\label{as0}
\Lambda_{K}(\varrho,\varphi)\equiv 0 \quad \forall K<L, \quad \Lambda_{L}(\varrho,\varphi)\not\equiv 0.
\end{gather}
Then, we take sufficiently large $n\geq L$ and consider a system obtained from \eqref{RPsiAv} by the truncation of the remainders $\widetilde \Lambda_n$ and $\widetilde \Omega_n$:
\begin{gather}\label{RPL}
  \frac{d \varrho}{d\tau}=\sum_{K=L}^n \tau^{-\frac{M+K}{N}}\Lambda_K(\varrho,\varphi),\quad
  \frac{d\varphi}{d\tau}=\sum_{K=0}^n \tau^{-\frac{M+K}{N}}\Omega_K(\varrho,\varphi), \quad \tau>\tau_0.
  \end{gather}
The functions $\Lambda_K(\varrho,\varphi)$, $\Omega_K(\varrho,\varphi)$ are defined by \eqref{LOm}.
In particular,  for all $K\in[0,M)$ we have $\Omega_K(\varrho,\varphi)\equiv \mathcal Q_K^0 \varrho$, $\Omega_{M+K}(\varrho,\varphi)=\Omega_{M+K}^0(\varphi)+\mathcal O(\varrho)$, $\Lambda_L(\varrho,\varphi)=\Lambda_L^0(\varphi)+\mathcal O(\varrho)$ as $\varrho\to 0$ uniformly for all $\varphi\in\mathbb R$.

Let us show that system \eqref{RPL} admits at least two different asymptotic regimes depending on the properties of the function $\Lambda_L(\varrho,\varphi)$. The first regime is associated with solutions such that $\varphi(\tau)\to{\hbox{\rm const}}$ as $\tau\to\infty$. Another mode is characterized by an unlimitedly growing phase difference.

Consider the following cases:
\begin{gather}
\label{as1}
\exists \, \varphi_0\in\mathbb R: \quad \Lambda_L^0(\varphi_0)=0, \quad \lambda_L:=\partial\Lambda_L^0(\varphi_0)\neq 0;\\
\label{as2}
\Lambda_L(\varrho,\varphi)\neq 0 \quad \forall\, (\varrho,\varphi)\in\mathbb R^2.
\end{gather}
The degenerate case with $\Lambda_L^0(\varphi)\equiv 0$ is not discussed in this paper.

We have the following.
\begin{Lem}
  Let assumptions \eqref{rc}, \eqref{as0} and \eqref{as1} hold. Then system \eqref{RPL} has a particular solution $\varrho_\ast(\tau)$, $\varphi_\ast(\tau)$ with asymptotic expansion in the form
  \begin{gather}
  \label{rpas}
  \varrho(\tau)=\sum_{K=0}^\infty \varrho_{K} \tau^{-\frac{M+K}{N}}, \quad \varphi(\tau)=\varphi_0+\sum_{K=1}^\infty \varphi_{K} \tau^{-\frac{K}{N}}, \quad \tau\to\infty,
\end{gather}
where $\varrho_K,\varphi_K={\hbox{\rm const}}$.
\end{Lem}
\begin{proof}
Substituting these series in system \eqref{RPL} and grouping the terms of the same power of $\tau$ yields $\varrho_0=-\Omega_M^0(\varphi_0)/\mathcal Q_0^0$ with $\mathcal Q_0^0=\omega_0 (h-1)>0$, and the chain of linear equations for the coefficients $\varrho_K$, $\varphi_K$, $K\geq 1$:
\begin{gather}\label{RSsys}
  \mathcal Q_0^0\varrho_{K}+\Big(\partial_\varphi\Omega_M^0(\varphi_0) +\delta_{N,2M}\frac{K}{N}\Big)\varphi_K=\mathfrak R_{K}, \quad \lambda_L\varphi_K=\mathfrak S_K,
\end{gather}
where the functions $\mathfrak R_K$ and $\mathfrak S_K$ are expressed through $\varrho_0,\varphi_0,\dots,\varrho_{K-1},\varphi_{K-1}$. For instance, if $M>3$, we have
\begin{eqnarray*}
    \mathfrak R_1   &=& -Q_1^0 \varrho_0-\Omega_{M+1}^0(\varphi_0),\\
    \mathfrak S_1   &=& -\Lambda_{L+1}^0(0,\varphi_0)-\delta_{N,L+1}\frac{M}{N}\varrho_0,\\
    \mathfrak R_2   &=& - Q_2^0 \varrho_0-Q_1^0 \varrho_1  -\Omega_{M+2}^0(\varphi_0) -\Big(\partial_\varphi\Omega_{M+1}^0(\varphi_0) +\delta_{N,2M+1}\frac{1}{N}\Big)\varphi_1-\partial_\varphi^2\Omega_M^0(\varphi_0)\frac{\varphi_1^2 }{2} ,\\
    \mathfrak S_2   &=& -\Lambda_{L+2}^0(0,\varphi_0)-\partial_\varphi\Lambda_{L+1}(0,\varphi_0) \varphi_1  - \partial^2_\varphi\Lambda_{L}(0,\varphi_0)\frac{\varphi_1^2}{2}-\delta_{N,L+1}\frac{M+1}{N} \varrho_1-\delta_{N,L+2}\frac{M}{N}\varrho_0,\\
    \mathfrak R_3   &=& - Q_3^0 \varrho_0-Q_2^0 \varrho_1-Q_1^0 \varrho_2  -\Omega_{M+3}^0(\varphi_0) -\partial_\varphi^2\Omega_{M+1}^0(\varphi_0)\frac{\varphi_1^2 }{2} -\partial_\varphi^2\Omega_{M}^0(\varphi_0) \varphi_1 \varphi_2 -\partial_\varphi^3\Omega_{M}^0(\varphi_0)\frac{\varphi_1^3 }{6} \\&&-\Big(\partial_\varphi\Omega_{M+2}^0(\varphi_0) +\delta_{N,2M+2}\frac{1}{N}\Big)\varphi_1-\Big(\partial_\varphi\Omega_{M+1}^0(\varphi_0) +\delta_{N,2M+1}\frac{1}{N}\Big)\varphi_2,\\
    \mathfrak S_3   &=& -\Lambda_{L+3}^0(0,\varphi_0) - \partial_\varphi\Lambda_{L+2}(0,\varphi_0)\varphi_1 - \partial_\varphi\Lambda_{L+1}(0,\varphi_0)  \varphi_2 - \partial^2_\varphi \Lambda_{L+1}(0,\varphi_0) \frac{\varphi_1^2}{2}  - \partial^2_\varphi \Lambda_{L}(0,\varphi_0) \varphi_1\varphi_2\\
                && - \partial^3_\varphi \Lambda_{L}(0,\varphi_0)\frac{\varphi_1^3}{6} - \delta_{N,L+1}\frac{M+2}{N}\varrho_2- \delta_{N,L+2}\frac{M+1}{N}\varrho_1 - \delta_{N,L+3}\frac{M}{N}\varrho_0.
\end{eqnarray*}
Note that system \eqref{RSsys} is solvable whenever $\lambda_L\neq 0$. The existence of a particular solution of system \eqref{RPL} with power-law asymptotics at infinity follows from~\cite{KF13,ANK89,LKJNMP08}.
\end{proof}

\begin{Lem}
  Let assumptions \eqref{rc}, \eqref{as0} and \eqref{as2} hold. Then the solutions of system \eqref{RPL} exit from any bounded domain in a finite time.
\end{Lem}
\begin{proof}
  Since $\Lambda_L(\varrho,\varphi)\neq 0$ for all $(\varrho,\varphi)\in\mathbb R^2$, it follows that for all $\delta>0$ there exist $C_1,C_2,C_3>0$ and $\tau_1\geq \tau_0$ such that
  \begin{gather*}
   \Big|\frac{d\varrho}{d\tau}\Big|\geq \tau^{-\frac{M+L}{N}} C_1, \quad \Big|\frac{d\varphi}{d\tau}\Big|\geq \tau^{-\frac{M}{N}} \Big(C_2|\varrho|-C_3\Big)
  \end{gather*}
for all $|\varrho|\leq \delta$, $\varphi\in\mathbb R$ and $\tau\geq \tau_1$. Integrating the last inequalities with respect to $\tau$ in the case $L<N-M$ yields
\begin{eqnarray*}
  |\varrho(\tau)-\varrho(\tau_1)|&\geq& \tilde C_1 \Big(\tau^{\frac{N-M-L}{N}}-\tau_1^{\frac{N-M-L}{N}}\Big), \\
  |\varphi(\tau)-\varphi(\tau_1)|&\geq& \tilde C_2 \Big(\tau^{\frac{2N-2M-L}{N}}-\tau_1^{\frac{2N-2M-L}{N}}\Big)-\tilde C_3 \Big(\tau^{\frac{N-M}{N}}-\tau_1^{\frac{N-M}{N}}\Big),
\end{eqnarray*}
as $\tau\geq \tau_1$ with positive constants $\tilde C_1=N C_1/(N-M-L)$, $\tilde C_2=N \tilde C_1 C_2/(2N-2M-L)$, $\tilde C_3=N(C_3+C_2|\varrho(\tau_1)|+\tilde C_1 C_2\tau_1^{(N-M-L)/N})/(N-M)$. Similar estimates hold in the case $L=N-M$. Hence, there exists $\tau_2\geq  \tau_1$ such that $|\rho(\tau_2)|+|\varphi(\tau_2)|\geq \delta$.
\end{proof}

Let us remark that the case of \eqref{as1} corresponds to a phase locking regime, when the phase $\phi(t)$ for solutions of the perturbed system \eqref{FulSys} is synchronized with the perturbations,  $\phi(t)-\varkappa^{-1}S(t)=\mathcal O(1)$ as $t\to\infty$, and the energy $E(t)$ increases significantly. The stability of such solutions is discussed in the next section.

In the case of \eqref{as2}, the phase $\phi(t)$ of solutions can significantly differ from that of the perturbations and the power-mode growth of the energy $E(t)$ does not occur. Such solutions correspond to a phase drifting. The qualitative analysis of this case requires special attention and is not discussed in this paper.

\section{Stability of phase locking}
\label{sec4}
Let $\varrho_\ast(\tau)$, $\varphi_\ast(\tau)$ be a solution of system \eqref{RPL} with asymptotics \eqref{rpas}. First, the stability of this solution in the truncated system is investigated. Then, we discuss its persistence in the full system \eqref{RPsiAv}.

Substitution $\varrho(\tau)=\varrho_\ast(\tau)+\hat\varrho(\tau)$, $\varphi(\tau)=\varphi_\ast(\tau)+\hat\varphi(\tau)$ into \eqref{RPL} gives the following system with a fixed point at $(0,0)$:
\begin{gather}
\label{rp0}
    \tau^{\frac{M}{N}}\frac{d\hat\varrho}{d\tau}=  {\bf  \Lambda} (\hat\varrho,\hat\varphi,\tau),\quad
    \tau^{\frac{M}{N}}\frac{d\hat\varphi}{d\tau}=  {\bf  \Omega} (\hat\varrho,\hat\varphi,\tau),
\end{gather}
where
\begin{eqnarray*}
  {\bf  \Lambda}(\hat\varrho,\hat\varphi,\tau)  & \equiv &\sum_{K=L}^{n}\tau^{-\frac{K}{N}}\Big( \Lambda_{K}(\varrho_\ast+\hat\varrho,\varphi_\ast+\hat\varphi)-
   \Lambda_{K}(\varrho_\ast,\varphi_\ast)\Big)\\
   &=&\tau^{-\frac{L}{N}}  \Big(\lambda_L \hat\varphi+\partial_{\hat\varrho}\Lambda_L(0,\varphi_0) \hat\varrho +\mathcal O(\Delta^2)\Big)\big(1+    \mathcal O( \tau^{-\frac{1}{N}}) \big),\\
 {\bf  \Omega}(\hat\varrho,\hat\varphi,\tau) & \equiv &\sum_{K=0}^{n}\tau^{-\frac{K}{N}}\Big( \Omega_{K}(\varrho_\ast+\hat\varrho,\varphi_\ast+\hat\varphi)-
   \Omega_{K}(\varrho_\ast,\varphi_\ast)\Big)\\
   &=&\Big(\mathcal Q_0^0 \hat\varrho +\tau^{-\frac{M}{N}}\big(\partial_{\hat\varphi}\Omega_M(0,\varphi_0)\hat\varphi+\mathcal O(\Delta^2)\big)\Big) \big(1+    \mathcal O( \tau^{-\frac{1}{N}}) \big)
\end{eqnarray*}
as $\tau\to\infty$ and $\Delta=\sqrt{\hat\varrho^2+\hat\varphi^2}\to 0$. Note that if $L<M$, then $\partial_{\hat \varrho}\Lambda_L(\hat\varrho,\hat\varphi)\equiv 0$.

\subsection{Linear analysis}
Consider the linearized system:
\begin{gather*}
 \tau^{\frac{M}{N}}\frac{d }{d\tau} \begin{pmatrix} \hat\rho \\ \hat\varphi \end{pmatrix}=
 {\bf a}(\tau)\begin{pmatrix} \hat\rho \\ \hat\varphi \end{pmatrix}, \quad
 {\bf a}(\tau)\equiv
  \begin{pmatrix*}[r]
            \displaystyle    \partial_\varrho {\bf  \Lambda}(0,0,\tau) & \displaystyle  \partial_\varphi { \bf \Lambda}(0,0,\tau)  \\
            \displaystyle  \partial_\varrho {\bf  \Omega}(0,0,\tau) & \displaystyle   \partial_\varphi {\bf  \Omega}(0,0,\tau)   \end{pmatrix*}.
\end{gather*}
The roots of the characteristic equation $|{\bf a}(t)-e{\bf I}|=0$ have the form:
\begin{gather*}
    e_\pm(\tau)=\frac{1}{2}{\hbox{\rm tr}}\,{\bf a}(\tau) \pm \frac{1}{2}
    \sqrt{
        \big({\hbox{\rm tr}}\,{\bf a}(\tau)\big)^2-4{\hbox{\rm det}} \,{\bf a}(\tau).
        }
\end{gather*}
Taking into account \eqref{rpas} and the structure of the functions $\Lambda_K$, $\Omega_K$, we obtain
\begin{gather*}
  {\hbox{\rm tr}} \,{\bf a}(\tau)=\mathcal O(\tau^{-\frac{M}{N}}), \quad {\hbox{\rm det}}\, {\bf a}(\tau)=-\lambda_L \mathcal Q_0^0   \tau^{-\frac{L}{N}}\big(1+\mathcal O(\tau^{-\frac{1}{N}})\big), \quad \tau\to\infty.
\end{gather*}
Since $0\leq L\leq \min\{N-M,2M-1\}$, we see that if $\lambda_L>0$, both eigenvalues $e_+(\tau)$,  $e_-(\tau)$ are real and have different signs:
\begin{gather*}
  e_\pm(\tau)= 2\tau^{-\frac{L}{2N}}\sqrt{\lambda_L \mathcal Q_0^0 }\big(1+\mathcal O(\tau^{-\frac{1}{N}})\big), \quad \tau\to\infty.
\end{gather*}
In this case, the fixed point $(0,0)$ of the linearized system is a saddle in the asymptotic limit, and the particular solution $\varrho_\ast(\tau)$, $\varphi_\ast(\tau)$ of system \eqref{RPL} is unstable.
\begin{Lem}
  Let assumptions \eqref{as0} and \eqref{as1} hold with $\lambda_L>0$. Then the solution $\varrho_\ast(\tau)$, $\varphi_\ast(\tau)$ of system \eqref{RPL} with asymptotics \eqref{rpas} is unstable.
\end{Lem}

In the opposite case, when $\lambda_L<0$, the eigenvalues are complex:
\begin{gather*}
  e_\pm(\tau)= 2i\tau^{-\frac{L}{2N}}\sqrt{|\lambda_L| \mathcal Q_0^0 }\big(1+\mathcal O(\tau^{-\frac{1}{N}})\big), \quad \Re e_\pm(\tau)=\mathcal O(\tau^{-\frac{M}{N}}), \quad  \tau\to\infty.
\end{gather*}
Hence, the fixed point $(0,0)$ is a centre in the asymptotic limit, and the linear analysis fails to determine the stability of the solution $\varrho_\ast(\tau)$, $\varphi_\ast(\tau)$ in the full nonlinear system (see, for example, \cite{OS20a}).

\subsection{Nonlinear analysis}
Assume that there exists $D\geq M$ such that
\begin{gather}\label{asg}
  \partial_{ \varrho} \Lambda_K( \varrho, \varphi)+\partial_{ \varphi} \Omega_K( \varrho, \varphi)\equiv 0 \quad \forall\, K< D,\\
  \nonumber\gamma_D:= \partial_{ \varrho} \Lambda_D( 0,\varphi_0)+\partial_{ \varphi} \Omega_D( 0,\varphi_0)\neq 0.
\end{gather}
We choose $n\geq L+D$ in \eqref{RPL}, then we have the following.

\begin{Lem}
  Let assumptions \eqref{rc}, \eqref{as0}, \eqref{as1}, \eqref{asg} hold with $\lambda_L<0$. Then then the solution $\varrho_\ast(\tau)$, $\varphi_\ast(\tau)$ of system \eqref{RPL} with asymptotics \eqref{rpas} is
  \begin{itemize}
    \item exponentially stable if $\gamma_D<0$ and $M+D<N$;
    \item polynomially stable if $\gamma_D+\frac{L}{N}<0$ and $M+D=N$;
    \item stable if $\gamma_D<0$, $L=0$, and $M+D>N$;
    \item unstable if $\gamma_D>0$.
  \end{itemize}
\end{Lem}
\begin{proof}
The proof is based on the construction of suitable Lyapunov function for system \eqref{rp0}. We first note that system \eqref{rp0} can be written as
\begin{gather}
  \label{RPHam1}
  \tau^{\frac{M}{N}}\frac{d\hat\varrho}{d\tau}=-\partial_{\hat \varphi} {  \Theta}(\hat\varrho,\hat\varphi,\tau), \quad \tau^{\frac{M}{N}}\frac{d\hat\varphi}{d\tau}= \partial_{\hat \varrho} {  \Theta}(\hat\varrho,\hat\varphi,\tau)+ {  \Upsilon}(\hat\varrho,\hat\varphi,\tau),
\end{gather}
with
\begin{eqnarray*}
  {  \Theta}(\hat\varrho,\hat\varphi,\tau)  \equiv   \int\limits_0^{\hat\varrho}{ \bf \Omega}(r,0,\tau)\,dr-\int\limits_0^{\hat\varphi}{\bf  \Lambda}(\hat\varrho,\theta,\tau)\,d\theta,\quad
  {  \Upsilon}(\hat\varrho,\hat\varphi,\tau) \equiv  \int\limits_0^{\hat\varphi}\Big(\partial_{\hat\varrho}{ \bf \Lambda}(\hat\varrho,\theta,\tau)+\partial_{\theta}{  \bf \Omega}(\hat\varrho,\theta,\tau)\Big)\,d\theta.
\end{eqnarray*}
It can easily be checked that $\Theta(\hat\varrho,\hat\varphi,\tau)=\Theta_L(\hat\varrho,\hat\varphi,\tau)+\mathcal O(\Delta^2 \tau^{-\frac{L+1}{N}})$, where
\begin{eqnarray*}
 \Theta_L(\hat\varrho,\hat\varphi,\tau)
    & \equiv &
        \sum_{K=0}^L\tau^{-\frac{K}{N}} \int\limits_0^{\hat\varrho}\Big(\Omega_K(r,\varphi_0) - \Omega_K(0,\varphi_0)\Big)\,dr -\tau^{-\frac{L}{N}}\int\limits_0^{\hat\varphi}\Big( \Lambda_L(\hat\varrho,\theta+\varphi_0)-\Lambda_L(0,\varphi_0)\Big)\,d\theta\\
    &=&
        \mathcal Q_0^0  \frac{\hat\varrho^2}{2} \Big(1+\mathcal O(\tau^{-\frac{1}{N}})\Big)  - \tau^{-\frac{L}{N}}\Big(\lambda_L \frac{\hat\varphi^2}{2}+\partial_{\hat\varrho}\Lambda_L(0,\varphi_0)\hat\varrho\hat\varphi+\hat\varphi\mathcal O(\Delta^2)   \Big),
\end{eqnarray*}
 as $\Delta\to 0$ and $\tau \to\infty$.
From \eqref{asg} it follows that
\begin{eqnarray*}
  \Upsilon(\hat\varrho,\hat\varphi,\tau)&\equiv&\sum_{K=M}^n\tau^{-\frac{K}{N}}\int\limits_0^{\hat\varphi} \Big(\partial_{\hat\varrho} \Lambda_K(\varrho_\ast+\hat\varrho,\varphi_\ast+\theta)+
  \partial_\theta \Omega_K(\varrho_\ast+\hat\varrho,\varphi_\ast+\theta)\Big)\,d\theta\\
  & =& \tau^{-\frac{D}{N}} \hat\varphi\big(\gamma_D +\mathcal O(\Delta)+\mathcal O(\tau^{-\frac{1}{N}})\big), \quad \Delta\to0, \quad \tau\to\infty.
\end{eqnarray*}
Here, the asymptotic estimates are uniform with respect to $(\hat\varrho,\hat\varphi,\tau)$ in the domain $\{(\hat\varrho,\hat\varphi,\tau)\in\mathbb R^3: \Delta\leq \Delta_\ast,\tau\geq \tau_\ast\}$ with some constants $\Delta_\ast>0$ and $\tau_\ast\geq \tau_0$.

Consider the combination
\begin{gather}\label{LF}
  V(\hat\varrho,\hat\varphi,\tau)=\Theta(\hat\varrho,\hat\varphi,\tau)+\tau^{-\frac{D}{N}} \gamma_D\frac{\hat\varrho\hat\varphi}{2}+\tau^{-\frac{D+L}{N}} \gamma_D \partial_{\hat\varrho}\Lambda_L(0,\varphi_0)   \frac{3\hat\varphi^2}{4\mathcal Q_0^0}
\end{gather}
as a Lyapunov function candidate for system \eqref{RPHam1}. It is easily shown that for all $\varepsilon\in(0,1)$ there exist $\Delta_1>0$ and $\tau_1\geq\tau_0$ such that
\begin{gather}\label{Veps}
 (1-\varepsilon) W_0(\hat\varrho,\hat\varphi,\tau) \leq V(\hat\varrho,\hat\varphi,\tau)\leq (1+\varepsilon) W_0(\hat\varrho,\hat\varphi,\tau)
\end{gather}
for all $(\hat\varrho,\hat\varphi,\tau)\in\mathbb R^3$ such that $\tau\geq \tau_1$ and $W_0(\hat\varrho,\hat\varphi,\tau_1)\leq\Delta_1^2$,
where
\begin{gather}\label{W0}
      W_0(\hat\varrho,\hat\varphi,\tau) \equiv \frac{1}{2}\Big(\mathcal Q_0^0\hat\varrho^2 +\tau^{-\frac L N} |\lambda_L|\hat\varphi^2\Big).
\end{gather}
Note that for any $\tau_\ast>0$ the following inequalities hold:
\begin{gather*}
\tau^{-\frac LN}W_0(\hat\varrho,\hat\varphi,\tau_\ast)\leq  W_0(\hat\varrho,\hat\varphi,\tau)\leq  W_0(\hat\varrho,\hat\varphi,\tau_\ast)
\end{gather*}
for all $\tau\geq \tau_\ast$ and $(\hat\varrho,\hat\varphi)\in\mathbb R^2$.

Calculating the total derivative of $V$ with respect to $\tau$ on the trajectories of system \eqref{RPHam1}, we obtain
\begin{gather}
    \label{DLF}
        \begin{split}
  \frac{dV}{d\tau}\Big|_{\eqref{RPHam1}}&\equiv  \partial_\tau V(\hat\varrho,\hat\varphi,\tau)+\tau^{-\frac{M}{N}}\partial_{\hat\varphi} V(\hat\varrho,\hat\varphi,\tau) \Upsilon(\hat\varrho,\hat\varphi,\tau)\\
  &=\tau^{-\frac{M+D}{N}}\gamma_D \left( \mathcal Q_0^0  \frac{\hat\varrho^2}{2}[1+\mathcal O(\tau^{-\frac{1}{N  }})] - \tau^{-\frac{L}{N}}\frac{\lambda_L}{2}
  [ \hat\varphi^2 +\hat\varphi\mathcal O(\Delta^2)  ]+\mathcal O\Big(\Delta^2\tau^{-\frac{L+1}{N}}\Big)\right)\\
  &=\tau^{-\frac{M+D}{N}}\gamma_D W_0(\hat\varrho,\hat\varphi,\tau) \big(1+\mathcal O(\Delta)+\mathcal O(\tau^{-\frac{1}{N}})\big)
  \end{split}
\end{gather}
as $\Delta\to 0$ and $\tau\to\infty$. Hence, there exist $0<\Delta_2\leq \Delta_1$ and $\tau_2\geq\tau_1$ such that
\begin{eqnarray}
\label{in1}
    \frac{dV}{d\tau}\Big|_{\eqref{RPHam1}}  \leq    -\tau^{-\frac{M+D}{N}}|\gamma_D|\left(\frac{ 1-\varepsilon }{ 1+\varepsilon }\right) V(\hat\varrho,\hat\varphi,\tau)\leq 0 & \quad & \text{if} \quad \gamma_D<0;\\
\label{in2}
    \frac{dV}{d\tau}\Big|_{\eqref{RPHam1}} \geq    \tau^{-\frac{M+D}{N}} \gamma_D \left(\frac{ 1-\varepsilon }{ 1+\varepsilon }\right) V(\hat\varrho,\hat\varphi,\tau)\geq 0 & \quad & \text{if} \quad\gamma_D>0
\end{eqnarray}
for all $(\hat\varrho,\hat\varphi,\tau)\in\mathbb R^3$ such that $\tau\geq \tau_2$ and $W_0(\hat\varrho,\hat\varphi,\tau_2)\leq\Delta_2^2$.

Let $\gamma_D<0$. Then integrating \eqref{in1} with respect to $\tau$ yields
\begin{gather*}
 \begin{split}
  0\leq W_0\big(\hat\varrho(\tau),\hat\varphi(\tau),\tau_2\big)\leq C_0 \tau^{\frac{L}{N}}
  \exp\left(-\frac{|\gamma_D|N}{N-M-D}\Big(\frac{1-\varepsilon}{1+\varepsilon}\Big)
   \tau^{1-\frac{M+D}{N}} \right),  \quad & \text{if} \quad M+D \neq N;\\
  0 \leq W_0\big(\hat\varrho(\tau),\hat\varphi(\tau),\tau_2\big)\leq C_0  \tau^{\frac{L}{N}+\gamma_D (\frac{1-\varepsilon}{1+\varepsilon} )},  \quad & \text{if} \quad {M+D}={N}
  \end{split}
\end{gather*}
as $\tau\geq \tau_2$ with some constant $C_0>0$, depending on $\Delta_2$ and $\tau_2$. Hence, if $M+D<N$, the fixed point $(0,0)$ of system \eqref{RPHam1} is exponentially stable. If $M+D=N$ and $\gamma_D+L/N<0$, then, by choosing $\varepsilon\in (0,1)$ small enough, we see that the fixed point $(0,0)$ is polynomially stable. The fixed point $(0,0)$ is stable if $M+D>N$ and $L=0$.

Let $\gamma_D>0$ and $\hat\varrho(\tau)$, $\hat\varphi(\tau)$ be a solution of system \eqref{RPHam1} with initial data $W_0(\hat\varrho(\tau_2),\hat\varphi(\tau_2),\tau_2)=\Delta_3^2< \Delta_2^2$. Integrating \eqref{in2} with respect to $\tau$ and taking into account \eqref{Veps} we obtain the following:
\begin{eqnarray*}
   W_0\big(\hat\varrho(\tau),\hat\varphi(\tau),\tau_2\big)\geq  \Delta_3^2 \Big(\frac{1-\varepsilon}{1+\varepsilon}\Big)
  \exp\left(\frac{\gamma_D N}{N-M-D}\Big(\frac{1-\varepsilon}{1+\varepsilon}\Big)
   \Big(\tau^{1-\frac{M+D}{N}} -\tau_2^{1-\frac{M+D}{N}}\Big)\right), & &  {M+D}\neq {N};\\
 W_0\big(\hat\varrho(\tau),\hat\varphi(\tau),\tau_2\big)\geq  \Delta_3^2 \Big(\frac{1-\varepsilon}{1+\varepsilon}\Big)
 \Big(\frac{\tau}{\tau_2}\Big)^{\gamma_D   (\frac{1-\varepsilon}{1+\varepsilon} )}, &&   {M+D}= {N}
\end{eqnarray*}
as $\tau\geq \tau_2$. Therefore, for all $\Delta_3>0$ there exists $\tau_3>\tau_2$ such that  $W_0\big(\hat\varrho(\tau),\hat\varphi(\tau),\tau_2\big)\geq \Delta_2^2$ as $\tau\geq \tau_3$. It follows that the fixed point $(0,0)$ of system \eqref{RPHam1} is unstable.

Returning to the variables $\varrho(\tau)$, $\varphi(\tau)$, we obtain the result of the Lemma.
\end{proof}

Note that the constructed Lyapunov function does not allow to prove the stability of the particular solution $\varrho_\ast(\tau)$, $\varphi_\ast(\tau)$ in the case of $\gamma_D<0$, $M+D\geq N$, $L>0$. Let us show that in this case there is at least a stability on a finite but asymptotically long time interval.

\begin{Lem}
  Let assumptions \eqref{rc}, \eqref{as0}, \eqref{as1}, \eqref{asg} hold with $\lambda_L<0$. If $\gamma_D<0$, $M+D\geq N$ and $L>0$, then the particular solution $\varrho_\ast(\tau)$, $\varphi_\ast(\tau)$ of system \eqref{RPL} with asymptotics \eqref{rpas} is stable on a finite but asymptotically long time interval.
\end{Lem}
\begin{proof}
    From \eqref{Veps} and \eqref{W0} it follows that for all $\delta\in (0,\Delta_2)$ there exists
   \begin{gather*}
     \Delta_\delta=\delta \Big(\frac{1-\varepsilon}{1+\varepsilon}\Big)^{\frac 12}\tau_2^{-\frac{L}{2N}}<\delta
   \end{gather*}
such that
\begin{gather*}
  \sup_{(\hat\varrho,\hat\varphi): W_0(\hat\varrho,\hat\varphi,\tau_2)\leq \Delta_\delta^2} V(\hat\varrho,\hat\varphi,\tau)\leq (1+\varepsilon) \Delta_\delta^2<(1-\varepsilon)\delta^2 \tau_2^{-\frac{L}{N}}\leq \inf_{(\hat\varrho,\hat\varphi): W_0(\hat\varrho,\hat\varphi,\tau/\tau_2)=\delta^2} V(\hat\varrho,\hat\varphi,\tau)
\end{gather*}
for all $1\leq\tau/\tau_2\leq \Gamma_\delta$ with $\Gamma_\delta=(\Delta_2/\delta)^{2N/L}$. Combining this with \eqref{in1}, we see that any solution of system \eqref{RPHam1} with initial data from $\{(\hat\varrho,\hat\varphi): W_0(\hat\varrho,\hat\varphi,\tau_2)\leq \Delta_\delta^2\}$ at $\tau=\tau_2$ satisfies the inequality $ W_0(\hat\varrho(\tau),\hat\varphi(\tau),\tau/\tau_2)<\delta^2$ as $1 \leq \tau/\tau_2 \leq \Gamma_\delta$, and $\Gamma_\delta\to\infty$ as $\delta\to 0$. Thus, the fixed point $(0,0)$ of system \eqref{RPHam1} and the particular solution $\varrho_\ast(\tau)$, $\varphi_\ast(\tau)$ of system \eqref{RPL} are stable on a finite but asymptotically long time interval.
\end{proof}

\subsection{Persistence of phase locking}
Let us show that if the particular solution $\varrho_\ast(\tau)$, $\varphi_\ast(\tau)$ is stable, the phase locking regime occurs in the full system \eqref{RPsiAv}.
We have the following.
\begin{Lem}
  Let assumptions \eqref{rc}, \eqref{as0}, \eqref{as1}, \eqref{asg} hold with $\lambda_L<0$. If any of the following conditions holds:
  \begin{itemize}
    \item $\gamma_D <0$ and $M+D< N$;
    \item $\gamma_D+\frac{L}{N}<0$ and $M+D=N$;
    \item $\gamma_D<0$, $L=0$, and $M+D>N$,
  \end{itemize}
then for all $\epsilon>0$ there exists $\delta>0$ and $\tau_s>0$ such that $\forall\,(R_s,\Psi_s)${\rm :} $ |R_s-\varrho_\ast(\tau_s)|<\delta$, $ |\Psi_s-\varphi_\ast(\tau_s)|<\delta$ the solution $R(\tau)$, $\Psi(\tau)$ of system \eqref{RPsiAv} with initial data $R(\tau_s)=\varrho_\ast(\tau_s)$, $\Psi(\tau_s)=\varphi_\ast(\tau_s)$ satisfies the inequality{\rm :} $|R(\tau)-\varrho_\ast(\tau)|+|\Psi(\tau)-\varphi_\ast(\tau)|<\epsilon$ for all $\tau>\tau_s$.
\end{Lem}
\begin{proof}
Substituting $R(\tau)=\varrho_\ast(\tau)+\hat R(\tau)$, $\Psi(\tau)=\varphi_\ast(\tau)+\hat \Psi(\tau)$ into \eqref{RPsiAv} yields
\begin{gather}
  \label{RPHam2}
  \tau^{\frac{M}{N}}\frac{d\hat R}{d\tau}=-\partial_{\hat \Psi} {  \Theta}(\hat R,\hat \Psi,\tau)+\widetilde{\bf \Lambda}_n(\hat R,\hat \Psi,\tau), \quad \tau^{\frac{M}{N}}\frac{d\hat \Psi}{d\tau}= \partial_{\hat R} {  \Theta}(\hat R,\hat\Psi,\tau)+ {  \Upsilon}(\hat R,\hat\Psi,\tau)+\widetilde{\bf \Omega}_n(\hat R,\hat \Psi,\tau),
\end{gather}
where $\widetilde{\bf \Lambda}_n\equiv \tau^{\frac{M}{N}}\widetilde \Lambda_n\big(\varrho_\ast(\tau)+\hat R,\varphi_\ast(\tau)+\hat\Psi,\tau\big)$, $\widetilde{\bf \Omega}_n\equiv \tau^{\frac{M}{N}} \widetilde \Omega_n\big(\varrho_\ast(\tau)+\hat R,\varphi_\ast(\tau)+\hat\Psi,\tau\big)$. It follows that
\begin{gather}\label{LOn}
    \begin{split}
        \widetilde{\bf \Lambda}_n(\hat R,\hat \Psi,\tau)=\mathcal O(\tau^{-\frac{n+1}{N}}), \quad
        \widetilde{\bf \Omega}_n(\hat R,\hat \Psi,\tau)=\mathcal O(\tau^{-\frac{n+1}{N}}), \quad \tau\to\infty
    \end{split}
\end{gather}
uniformly for all $(\hat R,\hat\Psi)\in\mathbb R^2$ such that $W_0(\hat R,\hat\Psi,1)\leq  \Delta_1^2$ with some $\Delta_1={\hbox{\rm const}}>0$. Note that the functions $\widetilde {\bf \Lambda}_n$ and $\widetilde {\bf \Omega}_n$ play the role of persistent disturbances of system \eqref{RPHam1}. Let us show that the particular solution $\varrho_\ast(\tau)$, $\varphi_\ast(\tau)$ of system \eqref{RPHam1} is stable with respect to these perturbations.

Using $V(\hat R,\hat \Psi,\tau)$ defined by \eqref{LF} as a Lyapunov function candidate for system \eqref{RPHam2}, we obtain
 \begin{gather}\label{DLF1}
  \frac{dV}{d\tau}\Big|_{\eqref{RPHam2}} \equiv   \frac{dV}{d\tau}\Big|_{\eqref{RPHam1}} +\tau^{-\frac{M}{N}}\Big(\partial_{\hat R} V(\hat R,\hat\Psi,\tau)\widetilde {\bf \Lambda}_n(\hat R,\hat\Psi,\tau)+\partial_{\hat\Psi} V(\hat R,\hat\Psi,\tau)\widetilde {\bf \Omega}_n(\hat R,\hat\Psi,\tau)\Big).
\end{gather}
It follows from \eqref{LOn} that there exist $C_\ast>0$ and $\tau_1\geq \max\{\tau_0,1\}$ such that
\begin{gather*}
  \partial_{\hat R} V \widetilde {\bf \Lambda}_n +\partial_{\hat\Psi} V \widetilde {\bf \Omega}_n\leq C_\ast \tau^{-\frac{n+1}{N}}\Big(\mathcal Q_0^0|\hat R|+\tau^{-\frac{L}{N}}|\lambda_L||\hat\Psi|\Big)
\end{gather*}
as $\tau\geq \tau_1$ and for all $(\hat R,\hat\Psi)\in\mathbb R^2$ such that $W_0(\hat R,\hat\Psi,\tau_1)\leq  \Delta_1^2$. Since $n\geq L+D$, we have
\begin{gather}
\label{est}
\begin{split}
  &|\hat R|\leq \frac{1}{2\delta}\big(\tau^{\frac{L}{2N}}\hat R^2+\tau^{-\frac{L}{2N}}\hat W_0(\hat R,\hat \Psi,\tau_1)\big), \quad |\hat\Psi|\leq \frac{1}{2\delta}\big(\tau^{\frac{L}{2N}}\hat \Psi^2+\tau^{-\frac{L}{2N}}\hat W_0(\hat R,\hat \Psi,\tau_1)\big),\\
 & \partial_{\hat R} V \widetilde {\bf \Lambda}_n +\partial_{\hat\Psi} V  \widetilde {\bf \Omega}_n \leq  \tau^{-\frac{D+1}{N}} W_0(\hat R,\hat\Psi,\tau)  \frac{ C_\ast}{\delta}\left(1+\frac{1}{\mathcal Q_0^0}+\frac{1}{|\lambda_L|}\right)
\end{split}
\end{gather}
as $\tau\geq\tau_1$ and for all $(\hat R,\hat\Psi)\in\mathbb R^2$ such that $\delta^2\leq W_0(\hat R,\hat\Psi,\tau_1)\leq  \Delta_1^2$ with some $\delta={\hbox{\rm const}}>0$. Thus, it follows from \eqref{DLF}, \eqref{DLF1} and \eqref{est} that for all $\varepsilon\in(0,1)$ there exists $\delta<\Delta_2\leq \Delta_1$ and $\tau_2\geq \tau_1$ such that
\begin{eqnarray}\label{deq}
 \frac{dV}{d\tau}\Big|_{\eqref{RPHam2}}  \leq    -\tau^{-\frac{M+D}{N}}|\gamma_D|\left(\frac{ 1-\varepsilon }{ 1+\varepsilon }\right) V(\hat R,\hat \Psi,\tau)\leq 0,
\end{eqnarray}
as $\tau\geq \tau_2$ and for all $(\hat R,\hat\Psi)\in\mathbb R^2$ such that  $\delta^2\leq W_0(\hat R,\hat\Psi,\tau_2)\leq  \Delta_2^2$.
Integrating \eqref{deq} in the case $M+D=N$ yields
\begin{gather}\label{W0est}
 \begin{split}
  0 \leq W_0\big(\hat R(\tau),\hat\Psi(\tau),\tau_2\big)\leq  \left(\frac{1+\varepsilon}{1-\varepsilon}\right) W_0\big(\hat R(\tau_2),\hat \Psi(\tau_2),\tau_2\big) \Big(\frac{ \tau}{\tau_2}\Big)^{\frac{L}{N}+\gamma_D (\frac{1-\varepsilon}{1+\varepsilon} )}
  \end{split}
\end{gather}
as $\tau\geq\tau_2$. By choosing $\varepsilon\in(0,1)$ small enough, we can ensure that $\gamma_D (1-\varepsilon)/(1+\varepsilon)+{L}/{N}<0$.  Hence, for all $\epsilon\in(0,\Delta_2)$ there exist $\delta=\epsilon\sqrt{(1-\varepsilon)/(1+\varepsilon)}/2<\epsilon$ such that any solution of system \eqref{RPHam2} starting from  $\{(\hat R,\hat \Psi): W_0(\hat R,\hat\Psi,\tau_2)\leq \delta^2\}$ at $\tau_s\geq \tau_2$ cannot leave the domain $\{(\hat R,\hat \Psi): W_0(\hat R,\hat\Psi,\tau_2)\leq \epsilon^2\}$ as $\tau>\tau_s$. Similar estimates hold in the case $M+D\neq N$.  Returning to the variables $(R,\Psi)$, we obtain the result of the Lemma.
\end{proof}

Define $\hat\gamma_D:=\gamma_D+\delta_{M+D,N}L/N$. Then we have the following:

\begin{Cor}
Let assumptions \eqref{rc}, \eqref{as0}, \eqref{as1}, \eqref{asg} hold with $\lambda_L<0$. If   $\hat\gamma_D<0$ and $M+D\leq N$, then for all $\varsigma\in (0,1)$ there exists $\Delta_s>0$ and $\tau_s>0$ such that $\forall\,(R_s,\Psi_s)${\rm :} $ |R_s-\varrho_\ast(\tau_s)|<\delta$, $ |\Psi_s-\varphi_\ast(\tau_s)|<\Delta_s$ the solution $R(\tau)$, $\Psi(\tau)$ of system \eqref{RPsiAv} with initial data $R(\tau_s)=\varrho_\ast(\tau_s)$, $\Psi(\tau_s)=\varphi_\ast(\tau_s)$ has the following estimates as $\tau\to\infty${\rm :}
\begin{gather}\label{ASRP}
  R(\tau)=\delta_{M+D,N}\mathcal O\big(\tau^{-\frac{\varsigma }{2}|\hat\gamma_D|}\big)+\mathcal O(\tau^{-\frac{M}{N}}), \quad \Psi(\tau)=\varphi_0+\delta_{M+D,N}\mathcal O\big(\tau^{-\frac{\varsigma}{2} |\hat\gamma_D|}\big)+\mathcal O(\tau^{-\frac{1}{N}}).
\end{gather}
\end{Cor}
\begin{proof}
  Let $M+D=N$. Then, by taking $\varepsilon=(1-\varsigma)|\hat \gamma_D|/(2|\gamma_D|-|\hat\gamma_D|)>0$ in \eqref{W0est}, we see that $W_0\big(\hat R(\tau),\hat\Psi(\tau),\tau_2\big)=\mathcal O(\tau^{-\varsigma |\hat\gamma_D|})$ as $\tau\to\infty$ for solutions of system \eqref{RPHam2} with initial data from  $\{(\hat R,\hat \Psi): W_0(\hat R,\hat\Psi,\tau_2)\leq \Delta_s^2\}$ with some $0<\Delta_s\leq \Delta_2$.  Similarly, if $M+D<N$, then from \eqref{deq} it follows that $W_0\big(\hat R,\hat\Psi,\tau_2\big)$ has exponentially decaying bound on the trajectories. Returning to the variables $(R,\Psi)$ and taking into account \eqref{rpas}, we obtain the corresponding asymptotic estimates.
\end{proof}

Combining this with Theorem~\ref{Th1}, we obtain the following:

\begin{Th}\label{Th3}
Let assumptions \eqref{fg}, \eqref{rc}, \eqref{as0}, \eqref{as1}, \eqref{asg} hold with $\lambda_L<0$ and some $\varkappa\in\mathbb Z_+$. If $\hat\gamma_D<0$ and $M+D\leq N$, then for all $\varsigma\in (0,1)$ there exist $t_s>0$ and $\mathcal D_s\subset \mathcal D(E_0)$ such that  for all $(x_s,y_s)\in\mathcal D_s$ the solution $x(t)$, $y(t)$ of system \eqref{FulSys} with initial data $x(t_s)=x_s$, $y(t_s)=y_s$ has the following estimates as $t\to\infty${\rm :}
\begin{eqnarray*}
  x(t)&=& t^{\frac{b}{(h-1)q}}  (2h)^{\frac{1}{2h}} c_\varkappa^{-1} X_0\big(\varkappa^{-1}S(t)+\Psi(\tau)\big) \Big(1+  t^{-\mu} R(\tau)\Big)^{\frac{1}{2h}}\Big(1+\mathcal O\big(t^{-\frac{b}{(h-1)q}}\big)\Big),\\
  y(t)&=&  t^{\frac{h b}{(h-1)q}} \sqrt 2 c_\varkappa^{-h}  Y_0\big(\varkappa^{-1}S(t)+\Psi(\tau)\big) \Big(1+  t^{-\mu} R(\tau)\Big)^{\frac{1}{2}}\Big(1+\mathcal O\big(t^{-\frac{b}{(h-1)q}}\big)\Big),
\end{eqnarray*}
where $R(\tau)$, $\Psi(\tau)$ have asymptotics \eqref{ASRP}, $\tau=t^\nu/\nu$, $\mu=(b/q-\sigma)/2>0$, $\nu=1+b/q$, and $X_0(\phi)$, $Y_0(\phi)$ is a $2\pi$-periodic solution of the system
\begin{gather*}
  \frac{2\pi}{\kappa}\partial_\phi X_0=Y_0, \quad \frac{2\pi}{\kappa}\partial_\phi Y_0=-X_0^{2h-1}, \quad \frac{X_0^{2h}}{2h}+\frac{Y_0^2}{2}=1.
\end{gather*}
\end{Th}

\section{Examples}
\label{secEX}
{\bf 1}. Consider again equation \eqref{ex} that satisfies the assumptions \eqref{fg} and \eqref{rc} with $h=2$, $a=b=1$, $q=3$, $l=p=0$, and $\sigma=-1$. From \eqref{munu} it follows that $\mu=2/3$, $\nu=4/3$, $M=4$, $N=8$. In this case, we have the following~\cite{IKetal16}:
\begin{gather*}
  \kappa=2\sqrt{2} {\hbox{\rm K}}\left(\frac{1}{2}\right), \quad
  X_0(\phi)=\sqrt 2 {\hbox{\rm cn}} \left(\frac{\kappa \phi}{ \pi\sqrt 2 };\frac{1}{2}\right), \quad
  Y_0(\phi)=\frac{2\pi}{\kappa} \partial_\phi  X_0(\phi),
\end{gather*}
where ${\hbox{\rm K}}(k)$ is the complete elliptic integral of the first kind, ${\hbox{\rm cn}}(t;k)$ is the Jacobi elliptic function~\cite{NIA90}.  Moreover, the $2\pi$-periodic function $X_0(\phi)$ admits the Fourier expansion:
\begin{gather*}
  X_0(\phi)=\sum_{j=1}^\infty  x_j \cos\big((2j-1)\phi\big), \quad x_j= \frac{4\pi\sqrt 2 }{\kappa} \hbox{\rm sech}\left((2j-1)\frac{\pi}{2}\right).
\end{gather*}

It is not hard to check that the corresponding averaged system \eqref{RPsiAv} takes the form
\begin{gather*}
  \frac{dR}{d\tau}=\tau^{-\frac{1}{2}}\sum_{K=0}^4\tau^{-\frac{K}{8}} \Lambda_K(R,\Psi)+\mathcal   O(\tau^{-\frac{9}{8}}),\quad
  \frac{d\Psi}{d\tau}=\tau^{-\frac{1}{2}}\sum_{K=0}^4\tau^{-\frac{K}{8}}  \Omega_K(R,\Psi)+\mathcal   O(\tau^{-\frac{9}{8}}), \quad \tau\to\infty,
\end{gather*}
 with $\tau=3 t^{4/3}/4$, $\Lambda_{1}\equiv \Lambda_{3}\equiv \Omega_{1}\equiv \Omega_{2}\equiv \Omega_{3}\equiv 0$,
\begin{align*}
   \Lambda_{0}  \equiv &  \nu^{-\frac 12} \left(\frac{c_\varkappa^2 B}{4   } \big\langle Y_0 (\Psi+\varkappa^{-1}\zeta )\cos\zeta\big\rangle_{\varkappa\zeta}-\frac{1}{3 }\right), \\ \Lambda_{2}\equiv & \nu^{-\frac 34} \frac{c_\varkappa^3 B}{4} \big\langle Y_1(\Psi+\varkappa^{-1}\zeta)\cos\zeta\big\rangle_{\varkappa\zeta},
  \\
   \Lambda_{4}  \equiv &   -R\nu^{-1}\left(\frac{c_\varkappa^2 B }{4}\big\langle Y_0(\Psi+\varkappa^{-1}\zeta)\cos\zeta\big\rangle_{\varkappa\zeta}-\frac{1}{3}\right) + \nu^{-1} \frac{2\omega_2 c_\varkappa^2}{3\omega_0}
  \\
     & + \nu^{-1}\frac{c_\varkappa^4 B}{4}\left(\big\langle Y_2(\Psi+\varkappa^{-1}\zeta)\cos\zeta\big\rangle_{\varkappa\zeta}+ \frac{2\omega_2}{\omega_0}\big\langle Y_0(\Psi+\varkappa^{-1}\zeta)\cos\zeta\big\rangle_{\varkappa\zeta}\right),
  \\
   \Omega_0 \equiv &  \omega_0 \nu^{-\frac 12}c_\varkappa^{-1} R, \\
    \Omega_4 \equiv & \nu^{- 1}\left( R \omega_2 c_\varkappa (\omega_2-1)-\frac{\omega_0 c_\varkappa^2 B}{4}\big\langle X_0(\Psi+\varkappa^{-1}\zeta)\cos\zeta\big\rangle_{\varkappa\zeta}\right).
\end{align*}
The parameters $\omega_k$ and $2\pi$-periodic functions $X_k(\phi)$, $Y_k(\phi)$ are defined in \eqref{omega}, \eqref{XY} with $h=2$. It is easily shown that $\langle Z(\Psi+\varkappa^{-1}\zeta)\cos\zeta\rangle_{\varkappa\zeta}\equiv \langle Z(\varkappa^{-1}\zeta) \cos\zeta\rangle_{\varkappa\zeta} \cos(\varkappa\Psi) +\langle Z(\varkappa^{-1}\zeta) \sin\zeta\rangle_{\varkappa\zeta} \sin(\varkappa\Psi)$, for any continuous $2\pi$-periodic function $Z(\phi)$.

Consider resonant solutions with $\varkappa=2m-1$,  $m\in\mathbb Z_+$. In this case, $\langle X_0 (\Psi+\varkappa^{-1}\zeta )\cos\zeta\rangle_{\varkappa\zeta}\equiv  (x_m/2) \cos(\varkappa\Psi)$ and $\langle Y_0 (\Psi+\varkappa^{-1}\zeta )\cos\zeta\rangle_{\varkappa\zeta}\equiv -(\pi \varkappa x_m /\kappa)\sin(\varkappa\Psi)$. Hence, the condition \eqref{as0} holds with $L=0$ and
\begin{gather*}
  \Lambda_0(R,\Psi)\equiv -\nu^{-\frac 12}\frac{\pi B \varkappa x_m c_\varkappa^2}{4\kappa}\Big(\sin(\varkappa\Psi) +\frac{s^2}{B d_\varkappa}\Big), \quad d_\varkappa:=\frac{27\pi^4 \varkappa^3{\hbox{\rm sech}}\big(\frac{\varkappa\pi}{2}\big)}{2\sqrt 2\kappa^4 }, \quad c_\varkappa=\frac{3  \varkappa \pi}{2 \kappa s}.
\end{gather*}
We see that if ${ s^{2}}/{|B|}<d_\varkappa$, there exists $\varphi_0$ such that $\Lambda_0(0,\varphi_0)=0$ and $\lambda_0=\partial_\Psi\Lambda_0(0,\varphi_0)<0$ (see Fig.~\ref{Fig2}, a). Moreover, it can easily be checked that the conditions \eqref{as1}, \eqref{asg} are satisfied with $D=4$  and $\gamma_4=\partial_R\Lambda_4(0,\varphi_0)+\partial_\Psi\Omega_4(0,\varphi_0)=-1/4$. It follows that equation \eqref{ex} satisfies the assumptions of Theorem~\ref{Th3} with $\hat\gamma_D=\gamma_D<0$, $M+D=N$. Hence, the phase locking regime with $\varkappa=2m-1$ is stable, and the resonant solutions of equation \eqref{ex} have the following asymptotics:
\begin{equation}\label{asympt}
\begin{aligned}
    &x(t)=t^{\frac{1}{3}} \sqrt 2   c_\varkappa ^{-1}  {\hbox{\rm cn}} \left(\frac{\kappa \phi(t)}{ \pi\sqrt 2 };\frac{1}{2}\right)\big(1+\mathcal O(t^{-\frac{1}{3}})\big),\\
    & I(t)= c_\varkappa^{-4}t^{\frac{4}{3}} \big(1+\mathcal O(t^{-\frac{1}{3}})\big), \quad   \phi(t)=\varkappa^{-1}S(t)+\varphi_0+\mathcal O(t^{-\frac{1}{6}}), \quad t\to\infty.
    \end{aligned}
\end{equation}

In particular, for primary resonant solutions with $\varkappa=1$ we have
\begin{eqnarray*}
  \varphi_0\in \Big\{ -\arcsin\Big(\frac{s^2}{B   d_1}\Big)+2\pi k, \ \ k\in\mathbb Z\Big\} & \quad & \text{if} \quad   B> B_1;\\
  \varphi_0\in  \Big\{\pi+ \arcsin\Big(\frac{s^2}{B   d_1}\Big)+ 2 \pi k, \ \ k\in\mathbb Z\Big\} & \quad & \text{if} \quad   B<-  B_1,
\end{eqnarray*}
where $B_1=s^2/d_1$  (see Fig.~\ref{Fig2}, b, c).

Note that $c_\varkappa^{-4}=\mathcal O(e^{-\varkappa \pi})$ as $\varkappa\to\infty$ for resonant solutions. Hence, if $\varkappa\gg 1$ and $B=\mathcal O(1)$, a significant increase of the energy can be achieved over a sufficiently long time interval. In this case, the admissible values of the parameters $(s,B)$ lie in a rather narrow domain: $s^2/B=\mathcal O(\varkappa^3 e^{-\varkappa \pi/2})$ as $\varkappa\to\infty$.
\begin{figure}
\centering
\subfigure[]
    {
     \includegraphics[width=0.42\linewidth]{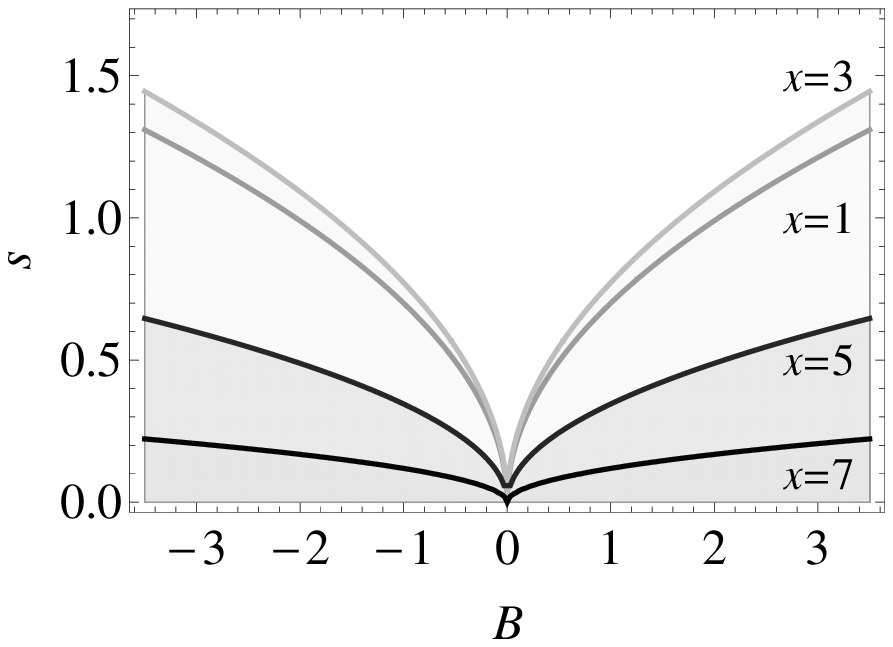}
    }
\subfigure[]
    {
     \includegraphics[width=0.42\linewidth]{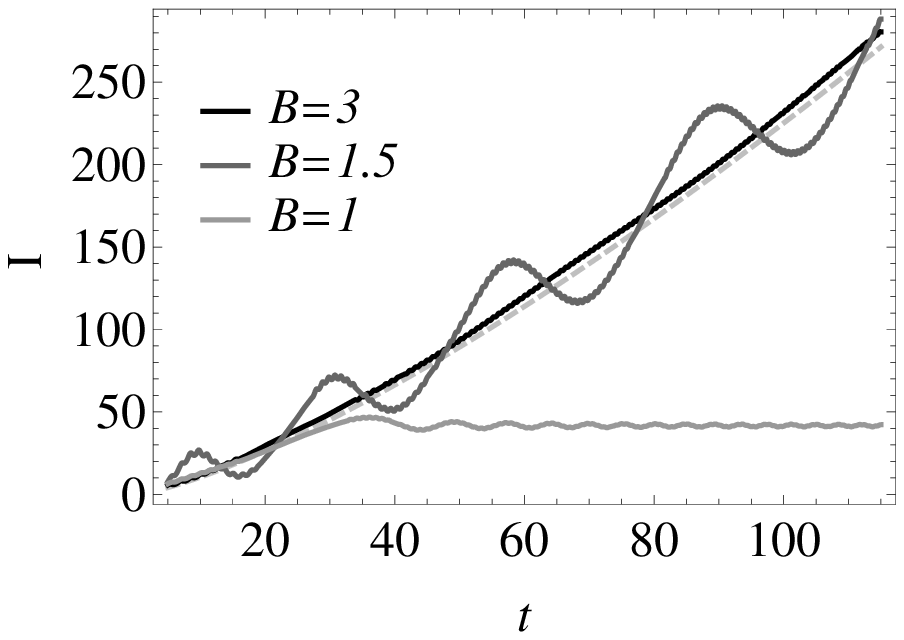}
    }\\
\subfigure[]
    {
     \includegraphics[width=0.3\linewidth]{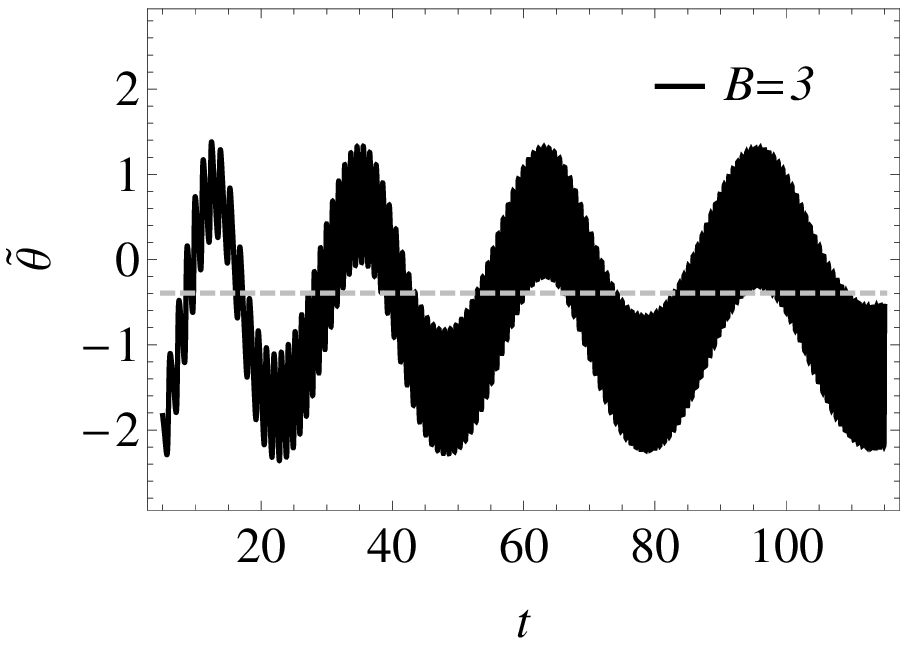} \ \
     \includegraphics[width=0.3\linewidth]{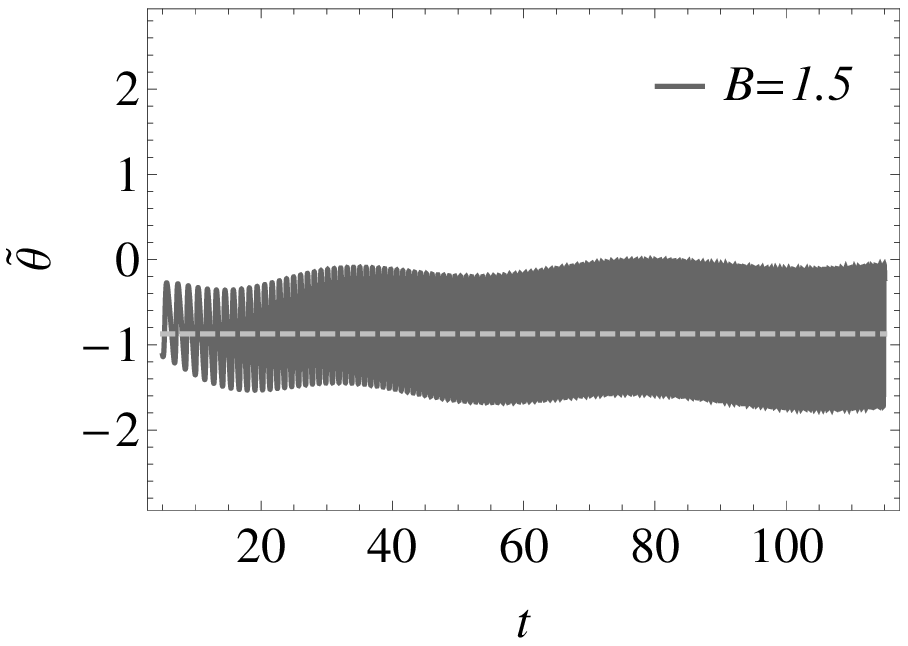} \ \
     \includegraphics[width=0.3\linewidth]{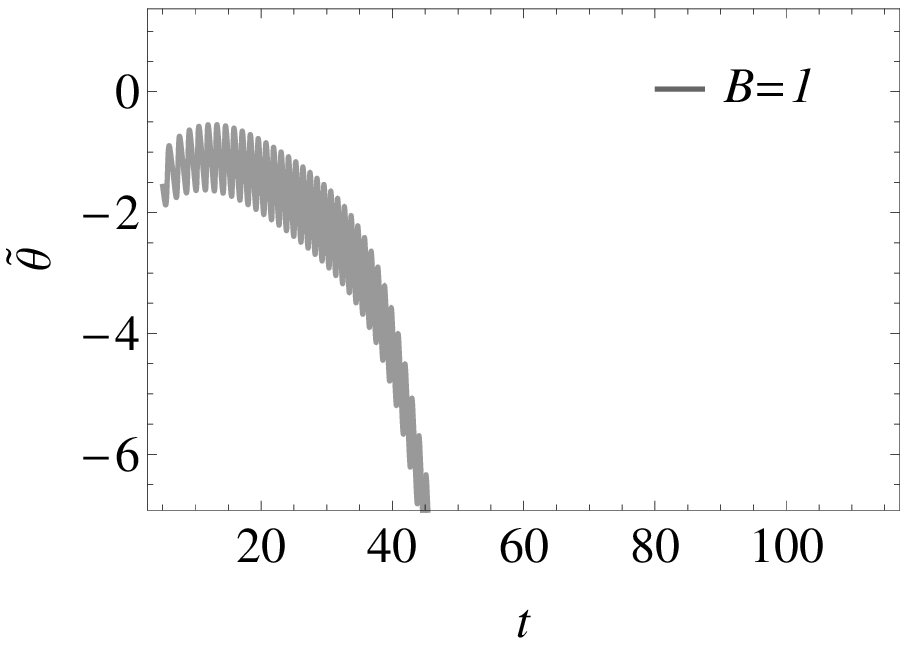}
    }
\caption{\footnotesize (a) Partition of the parameter plane $(B,s)$ for equation \eqref{ex} with different values of $\varkappa$. (b), (c) The evolution of $I(t)= H(x(t),\dot x(t))$ and $\tilde\theta(t)=\tilde\phi(t)-S(t)$, $\tan\tilde\phi(t)=-\dot x(t)/x(t)$ for solutions of \eqref{ex} with $q=3$, $b=1$, $s=0.75$, ($B_1\approx 1.15$) and different values of the parameter $B$. (b) The gray dashed curve corresponds to $ c_1^{-4} t^{\frac 43}$, $c_1^{-4}\approx 0.485$. (c) The gray dashed lines correspond to $\tilde\theta=\varphi_0$. } \label{Fig2}
\end{figure}

{\bf 2}. Consider the equation with a decreasing parametric perturbation:
\begin{gather}\label{ex2}
  \frac{d^2 x}{dt^2}-\big(1+Bt^{-\frac{a}{q}} \cos S(t)\big)x+x^3=0, \quad S(t)=s t^{1+\frac bq}.
\end{gather}
It is easy to verify that equation \eqref{ex2} in the variables $x,y= \dot x$ takes the form \eqref{FulSys} with $h=2$,  $U(x)\equiv x^4/4-x^2/2$, and satisfies \eqref{fg} with $l=0$, $p=1$, $f\equiv 0 $, $g\equiv B_{0,1,0}(S) x$, $ B_{0,1,0}(S)\equiv \cos S$. If $a+b\leq q$, then the condition \eqref{rc} is satisfied with $\sigma=-(a+b)/q$. Consider the case $q=3$, $a=2$, $b=1$. From \eqref{munu} it follows that $\mu=2/3$, $\nu=4/3$, $M=4$, $N=8$, and the corresponding averaged system \eqref{RPsiAv} takes the form
\begin{gather}\label{ex21}
  \frac{dR}{d\tau}=\tau^{-\frac{1}{2}}\sum_{K=0}^4\tau^{-\frac{K}{8}} \Lambda_K(R,\Psi)+\mathcal   O(\tau^{-\frac{9}{8}}),\quad
  \frac{d\Psi}{d\tau}=\tau^{-\frac{1}{2}}\sum_{K=0}^4\tau^{-\frac{K}{8}}  \Omega_K(R,\Psi)+\mathcal   O(\tau^{-\frac{9}{8}})
\end{gather}
as $\tau\to\infty$, with $\tau=(3/4)t^{4/3}$, $\Lambda_{1}\equiv \Lambda_{3}\equiv \Omega_{1}\equiv \Omega_{2}\equiv \Omega_{3}\equiv 0$,
\begin{align*}
   \Lambda_{0}  \equiv &  \nu^{-\frac 12} \left(\frac{c_\varkappa B}{4} \big\langle X_0  (\Psi+\varkappa^{-1}\zeta ) Y_0 (\Psi+\varkappa^{-1}\zeta )\cos\zeta\big\rangle_{\varkappa\zeta}-\frac{1}{3}\right), \\
   \Lambda_{2}\equiv & \nu^{-\frac 34} \frac{c_\varkappa^2 B}{4} \sum_{i+j=1} \big\langle X_i(\Psi+\varkappa^{-1}\zeta) Y_j(\Psi+\varkappa^{-1}\zeta)\cos\zeta\big\rangle_{\varkappa\zeta},\\
   \Lambda_{4}\equiv & \nu^{-1} \frac{c_\varkappa^2 B}{4} \sum_{i+j=2} \big\langle X_i(\Psi+\varkappa^{-1}\zeta) Y_j(\Psi+\varkappa^{-1}\zeta)\cos\zeta\big\rangle_{\varkappa\zeta}+\nu^{-1}\frac{R}{3}\\
   &+\nu^{-1}\frac{\omega_2 c_\varkappa^2}{\omega_0}\left(\frac{c_\varkappa B}{4}\big \langle X_0  (\Psi+\varkappa^{-1}\zeta ) Y_0 (\Psi+\varkappa^{-1}\zeta )\cos\zeta\big\rangle_{\varkappa\zeta}+\frac{2}{3}\right),\\
   \Omega_0 \equiv &  \omega_0 \nu^{-\frac 12}c_\varkappa^{-1} R, \\
    \Omega_4 \equiv & \nu^{- 1}\left((\omega_0z_2c_\varkappa^{-1}-\omega_2 c_\varkappa)R-\frac{\omega_0 c_\varkappa B}{4}\big\langle  X_0^2(\Psi+\varkappa^{-1}\zeta)  \cos\zeta\big\rangle_{\varkappa\zeta}\right).
\end{align*}
Consider a phase locking with $\varkappa=2$. In this case, system \eqref{ex21} satisfies the condition \eqref{as0} with $L=0$ and
\begin{gather*}
  \Lambda_0(R,\Psi)\equiv \nu^{-\frac{1}{2}} \frac{ c_2  B a_{11}}{4} \Big(\sin (2\Psi)+\frac{s}{B d_2}\Big),
\end{gather*}
where
\begin{gather*}
   a_{11}= \big\langle X_0  (\zeta ) Y_0 (\zeta )\sin(2\zeta)\big\rangle_{\zeta} = -\frac{\pi x_1^2}{2\kappa}-\frac{\pi}{\kappa}\sum_{j=1}^\infty x_j x_{j+1}, \quad d_2:=-\frac{9\pi a_{11}}{4\kappa}>0, \quad c_2=\frac{3\pi}{\kappa s}.
\end{gather*}
It follows that if  ${s}/{|B|}<  d_2$, there exists $\varphi_0$ such that $\Lambda_0(0,\varphi_0)=0$ and $\lambda_0=\partial_\Psi\Lambda_0(0,\varphi_0)<0$ (see Fig.~\ref{Fig3}, a). In particular,
\begin{eqnarray*}
  \varphi_0\in \Big\{-\frac12 \arcsin\Big(\frac{s}{B   d_2}\Big)+\pi k, \ \ k\in\mathbb Z\Big\} &\quad & \text{if} \quad   B> B_2;\\
  \varphi_0\in  \Big\{\frac{\pi}{2}+\frac 12 \arcsin\Big(\frac{s}{B   d_2}\Big)+\pi k, \ \ k\in\mathbb Z\Big\} & \quad & \text{if} \quad   B<-  B_2,
\end{eqnarray*}
where $  B_2=s/  d_2$.
Hence, the condition \eqref{as1} holds.
Since
\begin{gather*}
  \partial_\Psi \big \langle  X_0^2  (\Psi+\varkappa^{-1}\zeta )  \cos\zeta\big\rangle_{\varkappa\zeta}= \frac{2}{\omega_0} \big\langle X_0  (\Psi+\varkappa^{-1}\zeta ) Y_0 (\Psi+\varkappa^{-1}\zeta )\cos\zeta\big\rangle_{\varkappa\zeta} =\frac{\kappa a_{11} }{\pi}\sin(2\Psi),
\end{gather*}
we have $\gamma_4= \partial_R \Lambda_4(0,\varphi_0)+\partial_\Psi \Omega_4(0,\varphi_0)= -1/4$. Consequently, the condition \eqref{asg} is satisfied with $D=4$ and $\gamma_4\neq 0$. Therefore, equation \eqref{ex2} satisfies the assumptions of Theorem~\ref{Th3} with $\hat\gamma_D=\gamma_D<0$, $M+D=N$. Hence, the phase locking regime is stable and the solutions of \eqref{ex2} corresponding to parametric resonance have asymptotics \eqref{asympt} with $\varkappa=2$ (see Fig.~\ref{Fig3}, b, c).
\begin{figure}
\centering
\subfigure[]
    {
     \includegraphics[width=0.42\linewidth]{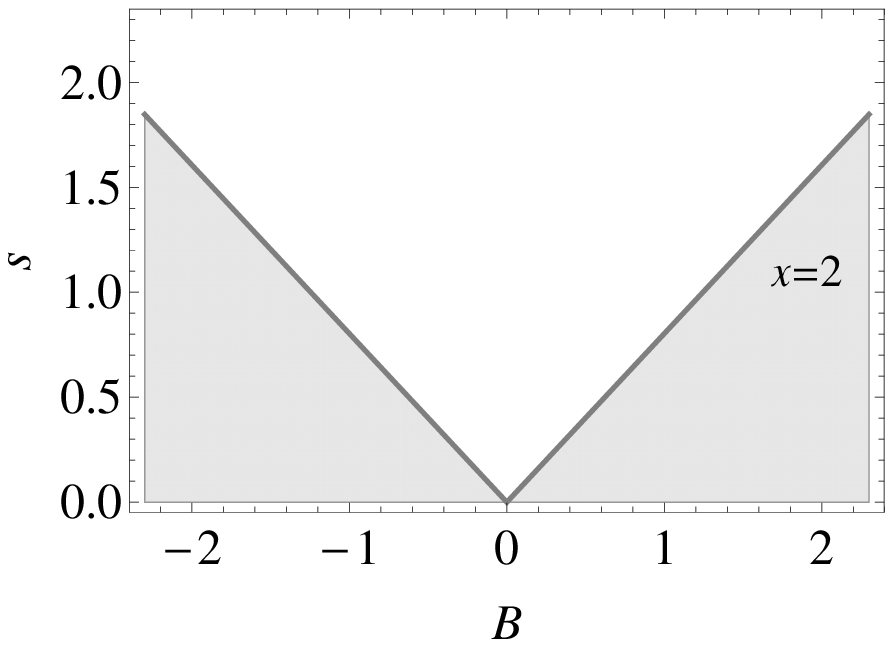}
    }
\subfigure[]
    {
     \includegraphics[width=0.42\linewidth]{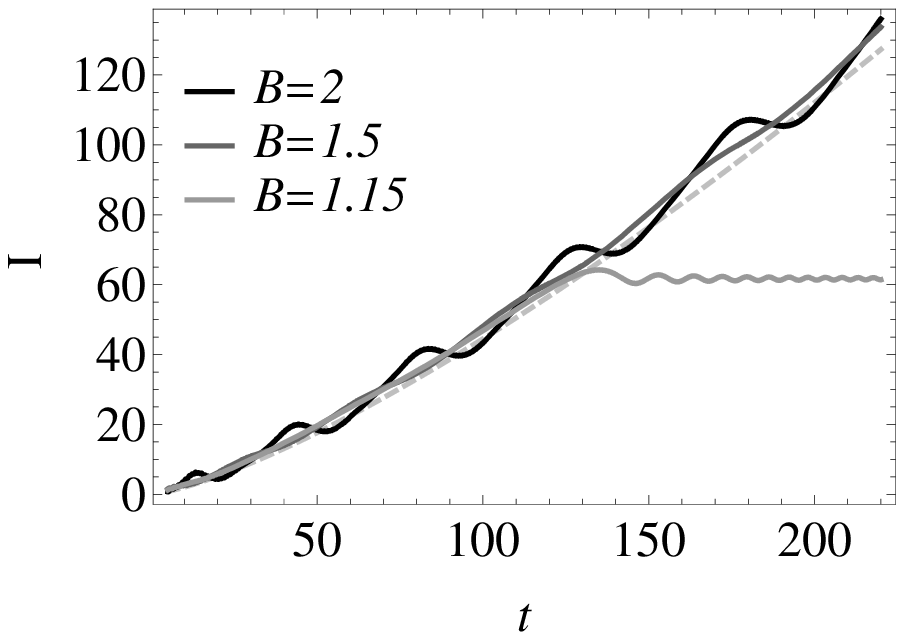}
    }\\
\subfigure[]
    {
     \includegraphics[width=0.3\linewidth]{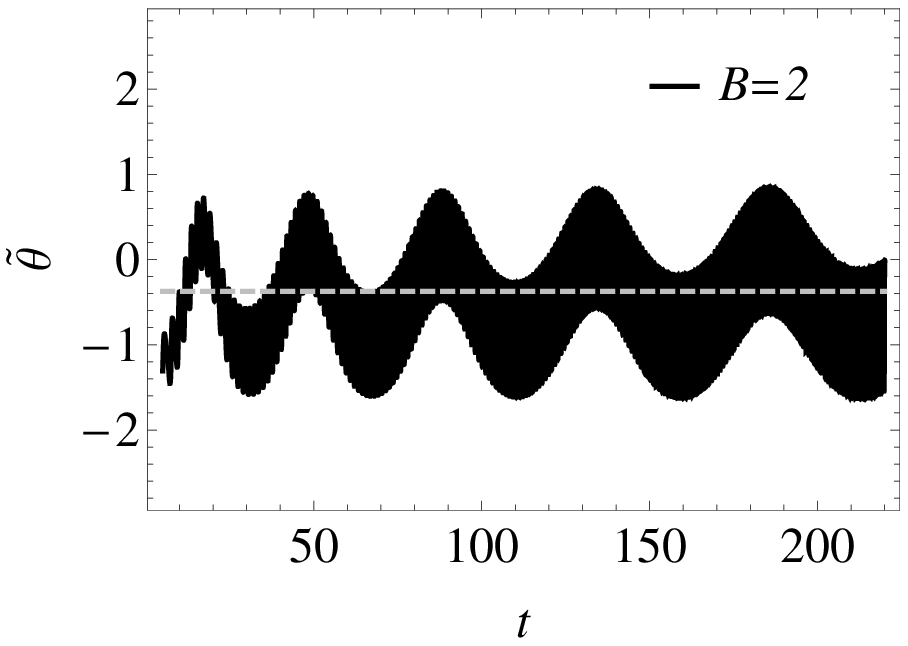} \ \
     \includegraphics[width=0.3\linewidth]{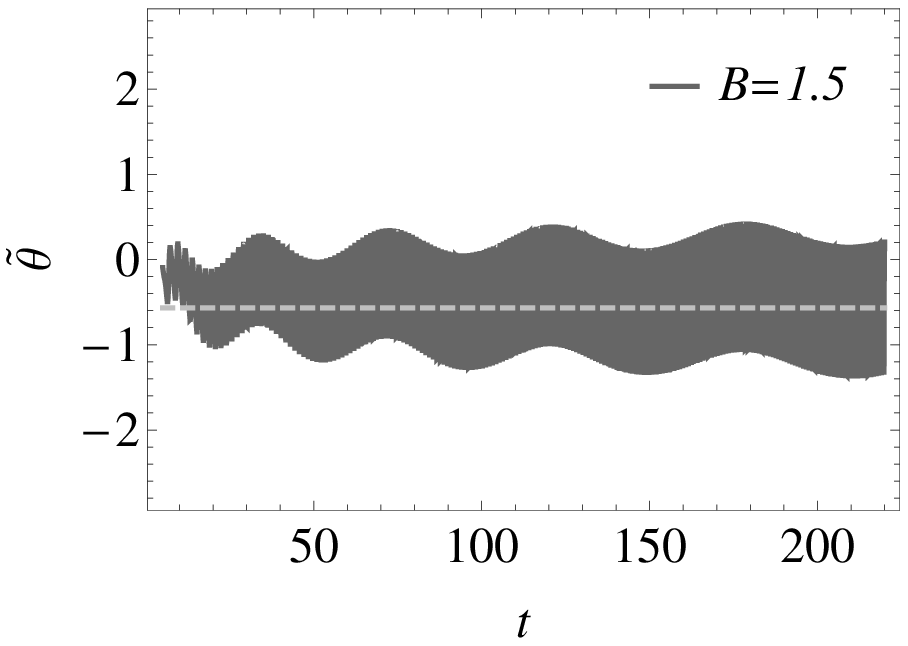} \ \
     \includegraphics[width=0.3\linewidth]{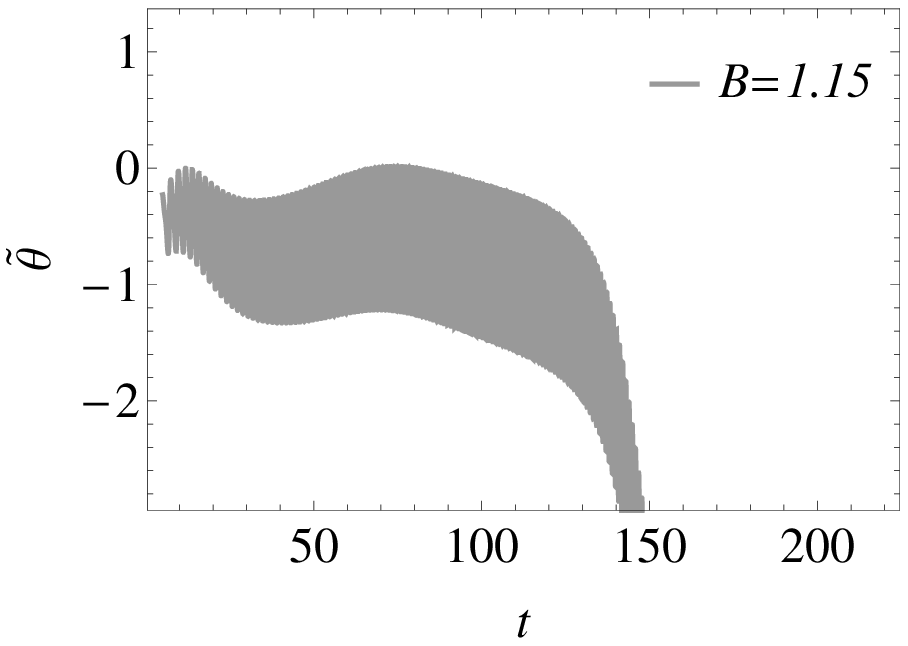}
    }
\caption{\footnotesize  (a) Partition of the parameter plane $(B,s)$ for equation \eqref{ex2} with $\varkappa=2$. (b), (c) The evolution of $I(t)= H(x(t),\dot x(t))$ and $\tilde\theta(t)=\tilde\phi(t)-S(t)/2$, $\tan\tilde\phi(t)=-\dot x(t)/x(t)$ for solutions of \eqref{ex2} with $q=3$, $a=2$, $b=1$, $s=1$, ($B_2\approx 1.24$) and different values of the parameter $B$. (b) The gray dashed curve corresponds to $c_2^{-4} t^{\frac 43}$, $c_2^{-4}\approx 0.095$. (c) The gray dashed lines correspond to $\tilde\theta=\varphi_0$. } \label{Fig3}
\end{figure}

{\bf 3}. Finally, consider a system with a nonlinear parametric perturbation and a weak nonlinear damping:
\begin{gather}\label{ex3}
  \frac{d^2 x}{dt^2}-x+\big(1-Bt^{-\frac{a}{q}} \cos S(t)\big)x^3+C t^{-\frac{a+1}{q}}x^2 \frac{dx}{dt}=0, \quad S(t)=s t^{1+\frac bq},
\end{gather}
where $B,C={\hbox{\rm const}}$, $C>0$. It is readily seen that \eqref{ex2} takes the form \eqref{FulSys} with $h=2$, $U(x)\equiv x^4/4-x^2/2$, and satisfies \eqref{fg} with $l=1$, $p=3$,  $f\equiv 0 $, $g\equiv B_{0,3,0}(S) x^3+ t^{-1/q}  B_{1,2,1}(S) x^2 y$, $B_{0,3,0}(S)\equiv B\cos S$, $B_{1,2,1}(S)\equiv -C$. Note that if $b<a$, then the condition \eqref{rc} holds with $\sigma=(2b-a)/q$. We take $a=q=3$ and $b=1$. Then $\sigma=-2/3$, $\mu=1/3$, $\nu=4/3$, $M=2$, $N=8$, and the corresponding averaged system \eqref{RPsiAv} takes the following form:
\begin{gather*}
  \frac{dR}{d\tau}=\tau^{-\frac{1}{4}}\sum_{K=0}^4\tau^{-\frac{K}{8}} \Lambda_K(R,\Psi)+\mathcal   O(\tau^{-1}),\quad
  \frac{d\Psi}{d\tau}=\tau^{-\frac{1}{4}}\sum_{K=0}^4\tau^{-\frac{K}{8}}  \Omega_K(R,\Psi)+\mathcal   O(\tau^{-1}), \quad \tau\to\infty,
\end{gather*}
with $\tau=(3/4) t^{4/3}$, $\Lambda_{0}\equiv\Lambda_{1}\equiv \Lambda_{3}\equiv \Omega_{1}\equiv \Omega_{2}\equiv \Omega_{3}\equiv 0$,
\begin{align*}
   \Lambda_{2}  \equiv &  -\nu^{-\frac 12} \frac{C}{4 c_\varkappa^2} \big\langle X_0^2  (\Psi+\varkappa^{-1}\zeta ) Y_0^2 (\Psi+\varkappa^{-1}\zeta )\big\rangle_{\varkappa\zeta}+\nu^{-\frac 12}\frac{B}{4c_\varkappa}\big\langle X_0^3  (\Psi+\varkappa^{-1}\zeta ) Y_0 (\Psi+\varkappa^{-1}\zeta )\cos\zeta\big\rangle_{\varkappa\zeta}, \\
   \Lambda_{4}\equiv &  -\nu^{-\frac 34} \frac{3C R}{4 c_\varkappa^2} \big\langle X_0^2  (\Psi+\varkappa^{-1}\zeta ) Y_0^2 (\Psi+\varkappa^{-1}\zeta )\big\rangle_{\varkappa\zeta}+\nu^{-\frac 34}\frac{B R}{2c_\varkappa}\big\langle X_0^3  (\Psi+\varkappa^{-1}\zeta ) Y_0 (\Psi+\varkappa^{-1}\zeta )\cos\zeta\big\rangle_{\varkappa\zeta}\\
   &-\nu^{-\frac 34}\frac{C}{2c_\varkappa}\big\langle X_0 Y_0 (X_1Y_0+X_0Y_1) \big\rangle_{\varkappa\zeta}+\nu^{-\frac 34}\frac{B}{4}\big\langle X_0^2(X_0Y_1+3X_1Y_0)\cos\zeta\big\rangle_{\varkappa\zeta}-\nu^{-\frac 34}\frac{1}{3},\\
   \Omega_0 \equiv &  \omega_0 \nu^{-\frac 14}c_\varkappa^{-1} R, \\
   \Omega_4 \equiv & \nu^{- \frac 34}\left(-2c_\varkappa\omega_2 R-\frac{\omega_0 B}{4 c_\varkappa}\big\langle  X_0^4(\Psi+\varkappa^{-1}\zeta)   \cos\zeta\big\rangle_{\varkappa\zeta}\right).
\end{align*}

It is easy to verify that for resonant solutions with $\varkappa=2$, the condition \eqref{as0} is satisfied with $L=2$ and
\begin{gather*}
  \Lambda_2(R,\Psi)\equiv \nu^{-\frac{1}{2}} \frac{ B  a_{31} \kappa s}{12\pi} \Big(\sin (2\Psi)+\frac{sC}{  B d_2 }\Big), \quad  d_2:=-\frac{3 \pi a_{31}}{\kappa v_{22} }>0,
\end{gather*}
where
$v_{22}\equiv\big\langle X_0^2  (\zeta ) Y_0^2 (\zeta )\big\rangle_{\zeta}>0$, $a_{31}\equiv  \big\langle X_0^3  (\zeta ) Y_0 (\zeta )\sin2\zeta\big\rangle_{\zeta}\approx x_1^2 a_{11}/2<0$. Hence, if  $sC/|B|<d_2$, the condition \eqref{as1} holds (see Fig.~\ref{Fig4}, a): there exists $\varphi_0$ such that $\Lambda_2(0,\varphi_0)=0$ and $\lambda_2=\partial_\Psi\Lambda_2(0,\varphi_0)<0$. In this case, we have
\begin{eqnarray*}
  \varphi_0\in\Big\{ -\frac 12 \arcsin\Big(\frac{s C}{B d_2}\Big)+\pi k, \ \ k\in\mathbb Z\Big\} &\quad & \text{if} \quad  B> B_2;\\
  \varphi_0\in \Big\{\frac{\pi}{2}+ \frac 12\arcsin\Big(\frac{s C}{B d_2}\Big)+ \pi k, \ \ k\in\mathbb Z\Big\}&\quad  & \text{if} \quad  B<-  B_2,
\end{eqnarray*}
where $B_2=s C/d_2$. Moreover, it can easily be checked that the assumption \eqref{asg} is satisfied with  $D=2$ and
\begin{gather*}
  \gamma_4\equiv \partial_R \Lambda_4(0,\varphi_0)+\partial_\Psi \Omega_4(0,\varphi_0)= -5 C v_{22} \nu^{-\frac 34} \Big(\frac{\kappa s}{6\pi}\Big)^2< 0.
\end{gather*}
Thus, by applying Theorem~\ref{Th3} with $\hat\gamma_D=\gamma_D<0$ and $M+D<N$, we see that the phase locking regime is stable and the resonant solutions of \eqref{ex3} have asymptotics \eqref{asympt} with $\varkappa=2$ (see Fig.~\ref{Fig4}, b, c).
\begin{figure}
\centering
\subfigure[]
    {
     \includegraphics[width=0.42\linewidth]{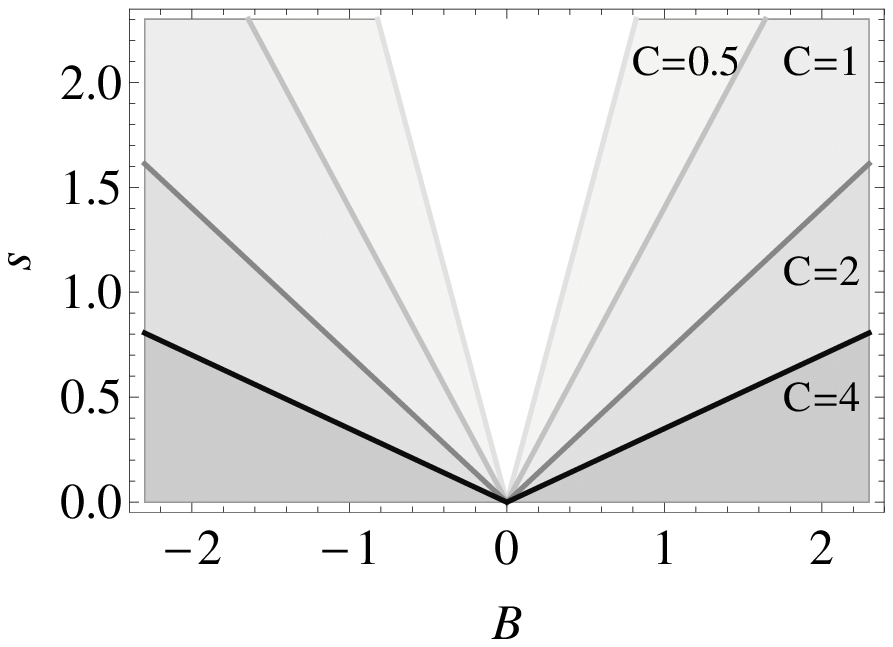}
    }
\subfigure[]
    {
     \includegraphics[width=0.42\linewidth]{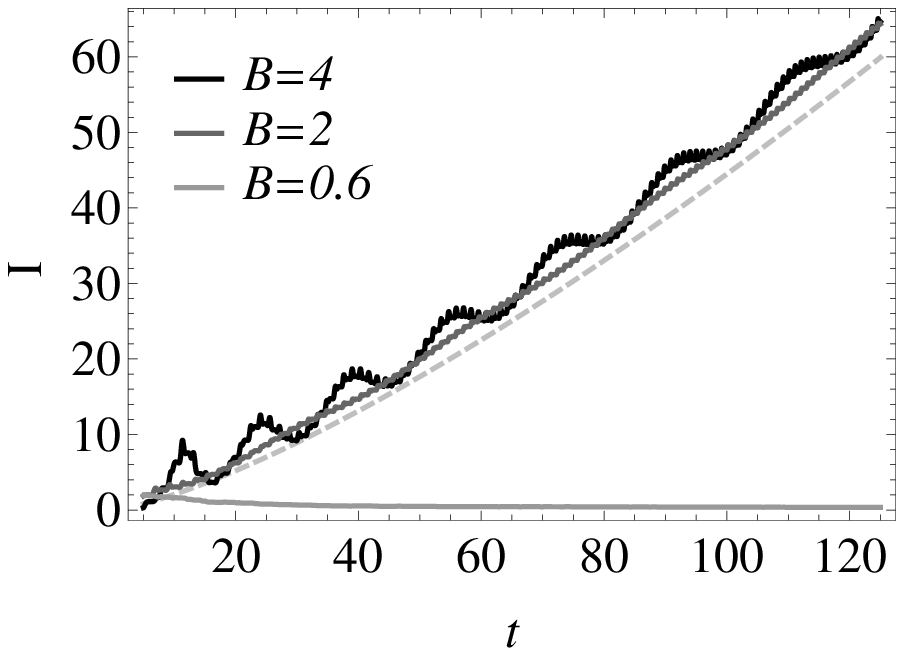}
    }\\
\subfigure[]
    {
    \includegraphics[width=0.3\linewidth]{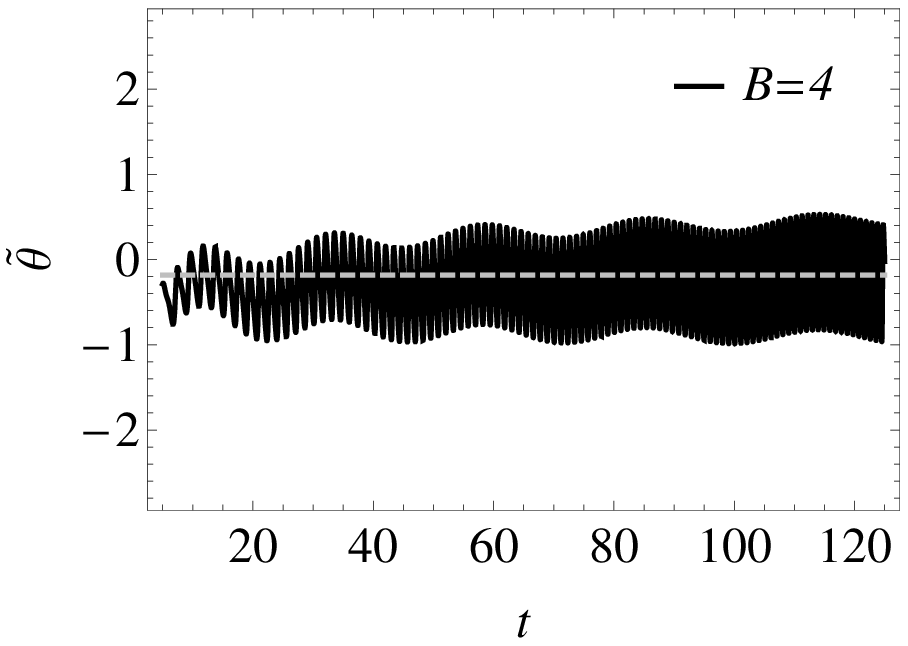} \ \
    \includegraphics[width=0.3\linewidth]{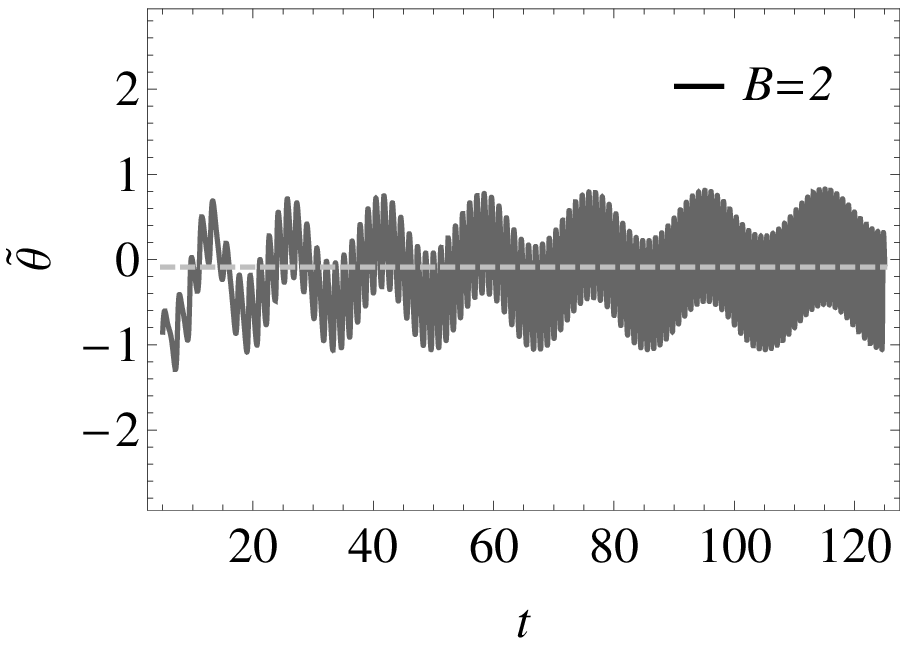} \ \
    \includegraphics[width=0.3\linewidth]{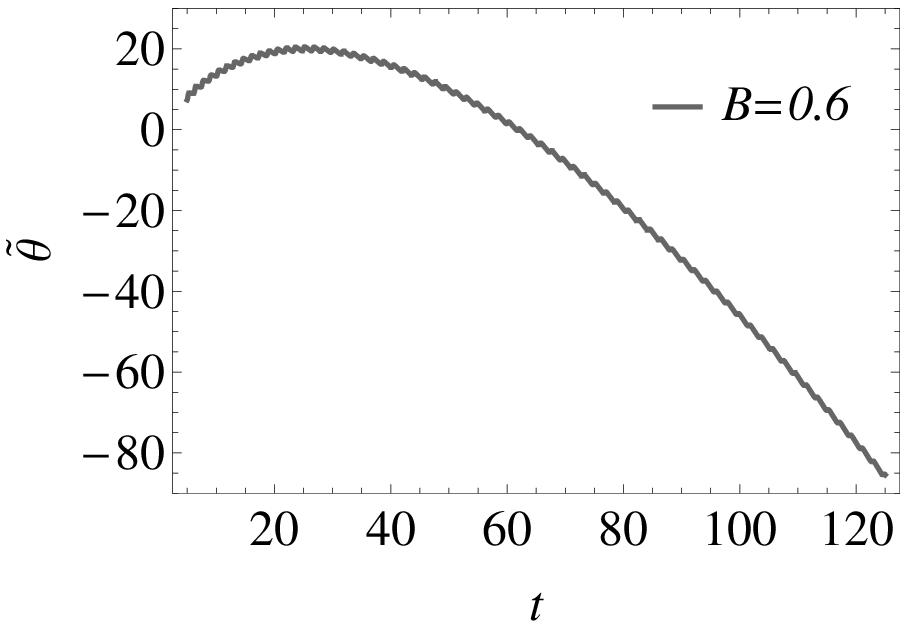}
    }
\caption{\footnotesize (a) Partition of the parameter plane $(B,s)$ for equation \eqref{ex2} with $\varkappa=2$ and different values of the parameter $C$. (b), (c) The evolution of $I(t)= H(x(t),\dot x(t))$ and $\tilde\theta(t)=\tilde\phi(t)-S(t)/2$, $\tan\tilde\phi(t)=-\dot x(t)/x(t)$ for solutions of \eqref{ex3} with $a=q=3$, $b=1$, $s=C=1$, ($ B_2\approx  0.71$) and different values of the parameter $B$. (b) The gray dashed curve corresponds to $ c_2^{-4} t^{\frac 43}$, $c_2^{-4}\approx 0.095$. (c) The gray dashed lines correspond to $\tilde\theta=\varphi_0$. } \label{Fig4}
\end{figure}

\section{Conclusion}

Thus, we have shown that decreasing chirped-frequency oscillatory perturbations of strongly nonlinear Hamiltonian systems in the plane can lead to the appearance of at least two different asymptotic regimes away from the equilibrium: a phase locking and a phase drifting. In the case of phase locking the energy of system can increase significantly and the phase of system is synchronised with the phase of the perturbation. We have described the conditions that guarantee the existence and stability of resonant solutions with growing energy. A violation of these conditions can lead to a phase drifting. Numerical examples show that in this case the energy of the perturbed system remains bounded. Note that such solutions have not been investigated in detail in this paper. This will be discussed elsewhere.

The results obtained show that it is possible to use vanishing in time perturbations for capture and holding of strongly nonlinear systems at resonance.

\section*{Acknowledgments}
Research is supported by the Russian Science Foundation grant 19-71-30002.

}

\end{document}